\documentclass[a4paper,11pt,oneside,reqno]{amsart}
\usepackage{amsmath,amssymb, geometry}
\usepackage{graphics,graphicx,epsfig}
\usepackage{bbm}
\DeclareGraphicsExtensions{.bmp,.png,.pdf,.jpg,.jpeg}

\theoremstyle{plain}

\begingroup

\newtheorem{thm}{Theorem}[section]
\newtheorem{cor}[thm]{Corollary}
\newtheorem{lemma}[thm]{Lemma}
\newtheorem{prop}[thm]{Proposition}
\newtheorem{thm-intro}{Theorem}
\newtheorem{prop-intro}{Proposition}
\newtheorem{cor-intro}{Corollary}
\newtheorem{dfn-intro}{Definition}
\newtheorem{que-intro}{Question}

\endgroup

\begingroup

\theoremstyle{definition}
\newtheorem{defn}[thm]{Definition}

\newtheorem{conje}{Conjecture}
\endgroup

\begingroup
\theoremstyle{remark}
\newtheorem{rem}[thm]{Remark}

\newtheorem{claim}{Claim}
\endgroup

\numberwithin{equation}{section}

\newcommand{\al}{\alpha}

\newcommand{\mc}[1]{\mathcal{#1}}
\newcommand{\mk}[1]{\mathfrak{#1}}
\newcommand{\pbs}{the Banach-Saks property}

\newcommand{\nrm}[1]{\|#1\|}
\DeclareMathOperator{\conv}{\mathrm{co}}

\newcommand{\defi}{\begin{defin}\rm}
\newcommand{\fdefi}{\end{defin}}
\newcommand{\eje}{\begin{ejemplo}}
\newcommand{\feje}{\end{ejemplo}}
\newcommand{\ejes}{\begin{ejemplos}}
\newcommand{\fejes}{\end{ejemplos}}
\newcommand{\lema}{\begin{lem}}
\newcommand{\flema}{\end{lem}}
\newcommand{\nota}{\begin{notas}\rm}
\newcommand{\fnota}{ \end{notas}}

\newcommand{\lclam}{\begin{lclaim}}
\newcommand{\flclam}{\end{lclaim}}

\newcommand{\ben}{\begin{enumerate}}
\newcommand{\een}{\end{enumerate}}
\newcommand{\bit}{\begin{itemize}}

\newcommand{\eit}{\end{itemize}}

\newcommand{\mr}[1]{\mathrm{#1}}

\newcommand{\casos}{\begin{itemize}}
\newcommand{\fcasos}{\end{itemize}\setcounter{cs}{1}}

\newcommand{\ro}{\varrho}

\newcommand{\conj}[2]{ \{ {#1}\,:\,{#2} \} }

\newcommand{\ou}{\omega_{1}}

\newcommand{\om}{\omega}
\newcommand{\Om}{\Omega}

\newcommand{\buit}{\emptyset}

\newcommand{\be}{\beta}
\newcommand{\de}{\delta}
\newcommand{\De}{\Delta}
\newcommand{\la}{\lambda}

\newcommand{\Sig}{\Sigma}

\newcommand{\vep}{\varepsilon}

\newcommand{\N}{{\mathbb N}}

\newcommand{\rest}{\upharpoonright}

\newcommand{\supp}{\mathrm{supp\, }}

\newcommand{\con}{\subseteq}

\newcommand{\prue}{\begin{proof}}
\newcommand{\fprue}{\end{proof}}
\makeindex

\begin{document}
\setlength{\baselineskip}{6mm}

\title{The convex hull of a Banach-Saks set}

\author[J. Lopez-Abad]{J. Lopez-Abad}
\address{J. Lopez-Abad \\ Instituto de Ciencias Matematicas (ICMAT).  CSIC-UAM-UC3M-UCM. C/ Nicol\'{a}s Cabrera 13-15, Campus Cantoblanco, UAM
28049 Madrid, Spain} \email{abad@icmat.es}

\author[C. Ruiz]{C. Ruiz}
\address{ C. Ruiz \\ Departamento de An{\'a}lisis Matem{\'a}tico, Facultad de
Matem{\'a}ticas, Universidad Complutense, 28040 Madrid, Spain \\}
\email{ Cesar\_Ruiz@mat.ucm.es}

\author[P. Tradacete]{P. Tradacete}
\address{P. Tradacete\\Mathematics Department\\ Universidad Carlos III de Madrid \\  28911 Legan\'es (Madrid). Spain.}
\email{  ptradace@math.uc3m.es }

\thanks{Support of the Ministerio de Econom\'{\i}a y Competitividad  under project MTM2008-02652 (Spain) and Grupo UCM 910346 is gratefully acknowledged.}



\begin{abstract}
A subset $A$ of a Banach space is  called Banach-Saks when every sequence in $A$ has a Ces{\`a}ro convergent
subsequence. Our interest here focusses on the following problem: is the convex hull of a Banach-Saks set
again Banach-Saks? By means of a combinatorial argument, we show that in general the answer is negative.
However, sufficient conditions are given in order to obtain a positive result.
\end{abstract}

\dedicatory{Dedicated to the memory of Nigel J. Kalton}

\maketitle

\section{Introduction}
A classical theorem of S. Mazur asserts that the convex hull of a compact set in a Banach space is again
relatively compact. In a similar way, Krein-\v{S}mulian's Theorem says that the same property holds for
weakly compact sets, that is, these sets have relatively weakly compact convex hull. There is a third
property, lying between these two main kinds of compactness, which is defined in terms of Ces\`aro
convergence. Namely, a subset $A $ of a Banach space $X$ is called Banach-Saks if every sequence in $A$ has a
Ces{\`a}ro convergent subsequence (i.e. every sequence $ (x_n)_n$ in $A$ has a subsequence $ (y_n)_n$ such
that the sequence of arithmetic means $((1/n)\sum_{i=1}^n y_{i})_n$   is norm-convergent in $X$).   In modern
terminology, as it was pointed out by H. P. Rosenthal \cite{Rosenthal}, this is equivalent to saying that  no
difference sequence in $A$ generates an $\ell_1$-spreading model.

The Banach-Saks property has its origins in the work of S. Banach and S. Saks \cite{BS}, after whom the property is
named. In that paper it was proved that the unit ball of $L_p$ ($1<p<\infty$) is a Banach-Saks set.
Recall that a Banach space is said to have the Banach-Saks property when its unit ball is a Banach-Saks set.
This property has been widely studied in the literature (see for instance \cite{Baernstein}, \cite{Beauzamy},
\cite{Farnum}) and more recently in \cite{ASS} and \cite{DSS}. Observe that since a
Banach space with the Banach-Saks property must be reflexive \cite{NW}, it is clear that neither $L_1$ nor
$L_\infty$ have this property. However, weakly compact sets in $L_1$ are Banach-Saks \cite{Szlenk}, and every
sequence of disjoint elements in $L_\infty$ is also a Banach-Saks set.

Since every compact set is Banach-Saks, and these sets are in turn weakly compact, taking into account both Mazur's and Krein-\v{S}mulian's results, it may seem reasonable to expect  that the convex hull of a Banach-Saks set is also Banach-Saks. We will show in Section  \ref{counterexample} that this is not the case in general. We present a canonical example consisting of the weakly-null unit basis
$(u_n)_n$ of a Schreier-like space $X_\mc F$ for a certain family of finite subsets $\mc F$ on $\N$ that we
call a $T$-family (see Definitions \ref{def-2} and \ref{ij4tijrigrf}). The role of the Schreier-like spaces and such families
is not incidental. There are several equivalent conditions to the
Banach-Saks property in terms of properties of certain families of finite subsets of $\N$ (see Theorem \ref{char1}), and in fact we
prove in Theorem \ref{erioeiofjioedf} that a possible counterexample must be of the form $X_\mc F$ for a
$T$-family $\mc F$. Therefore, an  analysis of the families of finite subsets of integers is needed to understand the Banach-Saks property.

The example of a $T$-family we present is influenced on a classical construction  P. Erd\H{o}s and A. Hajnal
\cite{ErHa} of a sequence of measurable subsets of the unit interval indexed by pairs of integers.
 These sequences of events behave in general in a completely different way than those indexed by integers, as
it can be seen, for example, in the work of D. Fremlin and M. Talagrand \cite{FrTa}.  Coming back to our
space, every subsequence of the basis $(u_n)_n$ has a further subsequence which is equivalent to the unit
basis of $c_0$, yet there is a block sequence of averages of $(u_n)_n$ generating an $\ell_1$-spreading
model. There is also the reflexive counterpart, either by considering a Baernstein space associated to $\mc
F$, or from a more general approach
 considering a Davis-Figiel-Johnson-Pelczynski interpolation space of $X_\mc F$.

As far as we know, the main question considered in this paper appeared explicitly in \cite{GG}, where the
authors also proved that every Banach-Saks set in the Schreier space has Banach-Saks convex hull. We will
  see in Theorem \ref{ioo34iji4jtr} that this fact can be further extended to Banach-Saks sets contained
in generalized Schreier spaces.

The paper is organized as follows: In Section 2 we introduce some notation, basic definitions and facts
concerning the Banach-Saks property, with a special interest on its combinatorial nature. In Section 3 several sufficient conditions are given for the stability of the Banach-Saks property under taking convex hulls. This includes the study of Banach-Saks sets in Schreier-like spaces $X_{\mc S_\al}$ defined from any generalized Schreier family $\mc S_\al$. Finally, in Section \ref{counterexample} we present a canonical example of a Banach-Saks set whose convex hull is not, as well as the corresponding reflexive version.


\section{Notation, basic definitions and facts}\label{Preliminaries}

We use standard terminology in Banach space theory from the monographs \cite{A-K} and \cite{Li-Tz}. Let us introduce now some basic concepts in infinite Ramsey theory, that will be  used
throughout this paper. Unless specified otherwise, by a family $\mathcal{F}$ on a set $I$  we mean a
collection of finite subsets of $I$. We denote infinite subsets by capital letters $M,N,P,\dots$, and finite
ones with $s,t,u,\dots$.   Given a family $\mc F$ on $\N$, and $M\con \N$, we define the \emph{trace} $\mc
F[M]$ of $\mc F$ in $M$ and the \emph{restriction} $\mc F\rest M$ of $\mc F$ in $M$ as
\begin{align*}
\mc F[M]:=&\conj{s\cap M}{s\in \mc F}, \\
\mc F\rest M:=&\conj{s\in \mc F}{s\con M},
\end{align*}
respectively.  A family $\mc F$ on $I$ is called compact, when it is compact with respect to the topology
induced by the product topology on $2^I$. The family $\mc F$ is pre-compact, or relatively compact, when the
topological closure of $\mc F$ consists only of finite subsets of $I$. The family $\mc F$ is
\emph{hereditary} when for every $s\con t\in \mc F$ one has that $s\in \mc F$. The $\con$-closure of $\mc F$
is the minimal hereditary family $\widehat{\mc F}$ containing $\mc F$, i.e. $\widehat{\mc F}:=\conj{t\con
s}{s\in \mc F}$. It is easy to see that  $\mc F$ is pre-compact if and only if $\widehat{\mc F}$ is compact.
Typical examples of pre-compact families are
\begin{align*}
{[I]}^n\,\,\,  :=& \{s\con I : \#s =n\},\\
{[I]}^{\leq n}  :=& \{s\con I : \#s \leq n\}, \\
{[I]}^{<\omega}  :=& \{s\con I : \#s < \infty\}.
\end{align*}
A natural procedure to obtain  pre-compact families is to consider, given  a relatively weakly-compact subset
$\mc K$ of $c_0$ and $\vep,\de>0$, the sets
\begin{align*}
\supp_{\vep}(\mc K):= &\conj{\supp_{\vep}x}{x\in  \mc K},\\
\supp_{\vep,+}(\mc K):= &\conj{\supp_{\vep,+}x}{x\in  \mc K},\\
\supp_{\vep}^{\de}(\mc K):=&\conj{\supp_{\vep,+}x}{x\in (\mc K)_\vep^\de},
\end{align*}
where $\supp_{\vep}x:=\conj{n\in \N}{|(x)_n|\ge \vep}$, $\supp_{\vep,+}x:=\conj{n\in \N}{(x)_n\ge \vep}$,
$(\mc K)_\vep^\de:=\conj{x\in \mc K}{\sum_{n\notin \supp_\vep x}|(x)_n|\le \de}$,  and $(x)_n$ denotes the
$n^\text{th}$ coordinate of $x$ in the canonical unit basis of $c_{00}$.

In particular, when $(x_n)_n$ is a weakly-convergent sequence to $x$ in some Banach space $X$, and $\mc M$ is
an arbitrary subset  of $B_{X^*}$ the family $\mc K:=\conj{(x^*(x_n-x))_n}{x^*\in \mc M}\con c_0$ is
relatively weakly-compact. Given $\vep,\de>0$ and  $\mc M\con B_{X^*}$, we define
\begin{align*}
\mc F_\vep((x_n)_n,\mc M):=&\supp_\vep(\mc K),\\
\mc F_\vep^\de((x_n)_n,\mc M):=&\supp_\vep^\de(\mc K).
\end{align*}
When $\mc M=B_{X^*}$ we will simply omit $\mc M$ in the terminology above.

Given $n\in \N$, a family $\mc F$ on $I$ is called \emph{$n$-large} in some $J\con I$ when for every infinite
$K\con J$  there is
  $s\in \mc F$ such that $\#(s\cap K)\ge n$. Or equivalently, when $\mc F[K]\not\con [K]^{\le n-1}$
for any $K\con J$. The family $\mc F$ is \emph{large} on $J$ when it is $n$-large on $J$ for every $n\in \N$.
Perhaps the first known example of a compact, hereditary  and large family is the Schreier family
$$\mc S:=\conj{s\con \N}{\# s \le \min s}.$$
Generalizing ideas used for families of sets, given  $\mc K\con c_0$ and $M\con \N$, we define $\mc
K[M]:=\conj{\mathbbm 1_M \cdot x}{x\in \mc K}$ as the image of $\mc K$ under the natural restriction to the
coordinates in $M$.
%
The following is a list of well-known results on compact families, commonly used by the specialist, which are necessary to understand most of the properties of Banach-Saks sets.
\begin{thm}\label{classif1}
Let $\mc K$ be a relatively weakly-compact subset of $c_0$, $\vep,\de>0$. Then there is an infinite subset
$M\con \N$ such that
\begin{enumerate}
\item[(a)] $\supp_\vep(\mc K [M])=\supp_\vep^\de(\mc K[M])$ and $\supp_\vep(\mc K [M])$ is hereditary, and
\item[(b.1)] either there is some $k\in \N$ such that $\supp_\vep(\mc K [M])=[M]^{\le k}$,
\item[(b.2)] or else ${}_*(\mc S\rest M):=\conj{s\setminus \{\min s\}}{s\in \mc S\rest M}\con \supp_\vep(\mc
K[M])$, and consequently  $ \supp_\vep(\mc K[M])$ is large in $M$.
\end{enumerate}
\end{thm}
The proofs of these facts are mostly based on the Ramsey property of  a particularly relevant type of
pre-compact families called \emph{barriers} on some set $M$, that were introduced by  C. ST. J. A.
Nash-Williams \cite{Nash}. These are families $\mc B$ on $M$ such that every further subset $N\con M$ has an
initial segment in $\mc B$, and such that there do not exist two different elements of $\mc B$ which are subsets one of the other.
Examples of barriers are $[\N]^{n}$, $n\in \N$, and the \emph{Schreier barrier} $\mk S:=\conj{s\in \mc
S}{\#s=\min s}$. As it was proved by Nash-Williams, barriers have the Ramsey property, and in fact provide a
characterization of it. The final ingredient is the fact that  if $\mc F$ is pre-compact, then there is a
trace $\mc F[M]$ of $\mc F$ which is the closure of a barrier on $M$ (we refer the reader to
\cite{AGR},\cite{Lo-To}).

\begin{defn}\label{def-1}
A subset $A$ of a Banach space $ X $ is a Banach-Saks set  (or has the Banach-Saks property) if  every
sequence $  (x_n)_n$ in $A$  has a \emph{Ces\`{a}ro}-convergent subsequence $(y_n)_n $, i.e.   the sequence of
averages $ ((1/n) \sum_{k = 1}^n y_k)_n$ is norm-convergent in $X$.

\end{defn}
It is easy to see that compact sets are Banach-Saks, that the Banach-Saks property is hereditary (every
subset of a Banach-Saks set is again Banach-Saks), it is closed under sums, and  that it is preserved under
the action of a bounded operator. It is natural to ask the following.

\begin{que-intro}\label{qu1}
Is the convex hull of a Banach-Saks set again a Banach-Saks set?
\end{que-intro}

Using the  localized notion of the Banach-Saks property, a space has the
Banach-Saks property precisely when its unit ball is a Banach-Saks set. A classical work by T. Nishiura and D.
Waterman \cite{NW} states that a Banach space with the Banach-Saks property is reflexive. Here is the local version of this fact.

\begin{prop}\label{BS->rwc}
Every Banach-Saks set is relatively weakly-compact.
\end{prop}
\prue
Let $A$ be a Banach-Saks subset of a Banach space $X$, and  fix a sequence $(x_n)_n$ in $A$. By Rosenthal's $\ell_1$ Theorem, there is a subsequence
$(y_n)_{n}$ of $(x_n)_n$ which is either equivalent to the unit basis of $\ell_1$ or weakly-Cauchy. The
first alternative cannot occur, since the unit basis of $\ell_1$ is not a Banach-Saks set. Let now $x^{**}\in
X^{**}$ be the $\mathrm{weak}^*$-limit of $(y_n)_{n}$.   Since $A$ is a Banach-Saks subset of $X$, there is a
further subsequence $(z_n)_{n}$ of $(y_n)_{n}$ which is Ces\`{a}ro-convergent to some $x\in X$. It follows  that
$x^{**}=x$, and consequently $(z_n)_{n}$ converges weakly to $x\in X$.
\fprue
 As the previous proof suggests, the unit basis of $\ell_1$ plays a very special role for the Banach-Saks
property. This is fully explained by the following characterization, due to H. P. Rosenthal \cite{Rosenthal}
and S. Mercourakis \cite{Mercourakis} in terms of the asymptotic notions of \emph{Spreading models} and
\emph{uniform weakly-convergence}.
\begin{dfn-intro}
Let $X$ be a Banach space and let $(x_n)_n$ be a sequence in $X$ converging weakly to $x\in X$. Recall that
 $(x_n)_n$ \emph{generates an $\ell_1$-spreading model}  when there is  $\de>0$ such that
\begin{equation}
\Big\|\sum_{n\in s} a_n (x_{n}-x)\Big\|\ge \de\sum_{n\in s} |a_n|
\end{equation}
for every   $s\con \N$ with $\# s\le \min s$ and every sequence $(a_n)_{n\in s}$ of scalars.


The sequence $(x_n)_n$ \emph{uniformly weakly-converges} to $x$ when for every $\vep>0$  there is an integer
$n(\vep)>0$ such that for every functional  $x^*\in B_{X^*}$
\begin{equation}\label{iijijdf}
\#(\conj{n\in \N}{|x^*(x_n-x)|\ge \vep})\le n(\vep).
\end{equation}

\end{dfn-intro}

The notion of $\ell_1$ spreading model  is   orthogonal   to the Banach-Saks property:   Suppose that
$(x_n)_n$ weakly-converges to $x$ and generates an $\ell_1$-spreading model. Let $\de>0$ be witnessing that.
Set $y_n=x_n-x$ for each $n$. Since $\nrm{y_n}\ge \de$ for all $n$, it follows by Mazur's Lemma that there is
a subsequence $(z_n)_{n}$ of $(y_n)_n$ which is a 2-basic sequence. We claim that no further subsequence of
$(z_n)_{n}$ is Ces\`{a}ro-convergent: Fix an arbitrary subset $s\con \N$ with even cardinality. Then the upper
half part $t$ of $s$ satisfies that $\#t\le \min t$. So, using also that $(z_n)_{n}$ is $2$-basic,
\begin{equation}
\left\|\frac1{\#s}\sum_{n\in s}z_n\right\|\ge \frac{1}{2}\left\|\frac{1}{\#s}\sum_{n\in t}z_n\right\|\ge \frac{\de}{2}\frac{\#t}{\#s}=\frac{\de}{4}.
\end{equation}
This immediately gives that no subsequence of  $(z_n)_n$ is Ces\`{a}ro-convergent to 0.

On the other hand if $(x_n)_n$ is uniformly weakly-convergent to some $x$, then every subsequence of
$(x_n)_n$ is Ces\`{a}ro-convergent (indeed these conditions are equivalent \cite{Mercourakis}): Suppose that
$(y_n)_n$ is a subsequence of $(x_n)_n$. Now for each $\vep>0$ let
 $n(\vep)$ be witnessing that \eqref{iijijdf} holds.   Set $z_n=y_n-x$ for each $n$.  Now suppose that $s$ is an
 arbitrary finite subset of $\N$ with cardinality $\ge n(\vep)$. Then, given $x^*\in B_{X^*}$, and setting $t:=\conj{n\in s}{|x^*(z_n)|\ge \vep}$,
 we have that
 \begin{equation}
\left| x^*(\frac{1}{\#s}\sum_{n\in s}z_n)\right| \le \frac{1}{\#s}\sum_{n\in t}|x^*(z_n)|+ \frac{1}{\#s}\sum_{n\in s\setminus t}|x^*(z_n)|\le
\frac{n(\vep)}{\#s}C+\vep.
 \end{equation}
Hence,
\begin{equation}
\left\|\frac{1}{\#s}\sum_{n\in s}z_n\right\| \le \frac{n(\vep)}{\#s}C+\vep.
 \end{equation}
This readily implies that $(z_n)_n$ is Ces\`{a}ro-convergent to 0, or, in other words, $(y_n)_n$ is
Ces\`{a}ro-convergent to $x$. Next result summarizes the relationship between these three notions.

\begin{thm}\label{char1}
Let $A$ be an arbitrary subset of a Banach space $X$. The following are equivalent:
\begin{enumerate}
\item[(a)] $A$ is a Banach-Saks subset of $X$.
\item[(b)] $A$ is relatively weakly-compact and for every weakly-convergent sequence in $A$ it never generates an  $\ell_1$-spreading model.
\item[(c)] $A$ is relatively weakly-compact and for every weakly-convergent sequence $(x_n)_n$ in $A$ and every $\vep>0$   the family
$\mc F_\vep((x_n)_n)$ is not large in $\N$.
\item[(d)] $A$ is relatively weakly-compact and for every weakly convergent sequence $(x_n)_n$ in $A$ there is some   norming set $\mc
N$ such that for every $\vep>0$ the family $\mc F_\vep((x_n)_n,\mc N)$ is not  large.
\item[(e)] For every sequence $(a_n)_n$ in $A$ there is a subsequence $(b_n)_n$ and some norming set $\mc N$ such
that for every $\vep>0$ there is $m\in \N$ such that $\mc F_\vep((b_n)_n,\mc N)\con [\N]^{\le m}$.

\item[(f)] Every sequence in $A$ has a uniformly weakly-convergent subsequence.
\end{enumerate}
\end{thm}
Recall that a $\la$-norming set, $0<\la\le 1$ is a subset $\mc N\con B_{X^*}$ such that
$$\la\nrm{x}\le  \sup_{f\in \mc N}{|f(x)|} \text{ for every $x\in X$}.$$
The subset $\mc N\con B_{X^*}$ is norming when it is $\la$-norming for some $0<\la\le 1$. Note we could rephrase (e) as saying that the sequence $(b_n)_n$ is uniformly weakly-convergent with respect to $\mc N$.

The equivalences between  (a) and (b), and between (a) and (f) are due to Rosenthal \cite{Rosenthal} and
Mercourakis \cite{Mercourakis}, respectively. For the sake of completeness, we give now hints of the proof of
Theorem \ref{char1} using, mainly, Theorem \ref{classif1}:

(a) implies (b) because we have already seen that if a sequence $(x_n)_n$ converges weakly to $x$, generates
an $\ell_1$-spreading model  and is such that $(x_n-x)_n$ is basic, then it does not have Ces\`{a}ro-convergent
subsequences.  We prove that (b) implies (c)   by using   Theorem \ref{classif1}. Let $(x_n)_n$ be a weakly
convergent sequence in  $A$ with limit $x$, and let us see that $\mc F_\vep((x_n))$ is not large for any
$\vep>0$.  Otherwise, by Theorem \ref{classif1}, there is some $M$ such that
$${}_*(\mc S\rest M)\con  \mc F_\vep^\de((x_n)_n)[M]=\mc F_\vep((x_n)_{n})[M].$$
Set $y_n:=x_n-x$ for each $n\in M$. It follows that $(y_n)_{n\in M}$  is a non-trivial weakly-null sequence,
hence by Mazur's Lemma, there is $N\con M$ such that $(y_n)_{n\in N}$ is a 2-basic sequence.  We claim that
then $(y_n)_{n\in N}$ generates an $\ell_1$-spreading model, which is impossible: Let $s\in \mc S\rest N$,
and let $(\la_k)_{k\in s}$ be a sequence of scalars. Let $t\con s$ be such that  $\la_k \cdot \la_l\ge 0$ for
all $k,l\in t$, $|\sum_{k\in t} \la_k|\ge 1/4\sum_{k\in s}|\la_k|$ and $t\in {}_*(S\rest N)$. Then let
$x^*\in B_{X^*}$ be such that
$$\text{$x^*(y_n)\ge \vep$ for $n\in t$, and $\sum_{n\in M\setminus t}|x^*(y_n)|\le
\frac\vep4$.}$$
It follows that
\begin{align*}
\left\|\sum_{k \in s} \la_k y_k\right\|\ge & \left|x^*(\sum_{k\in s}\la_k y_k)\right|\ge \left|\sum_{k\in t}\la_k x^*( y_k)\right|
-\frac{\vep}4\max_{k\in s}|\la_k|\ge \vep
\left|\sum_{k\in t}\la_k\right| -\frac{\vep}4\max_{k\in s}|\la_k|\\
\ge & \frac\vep4 \sum_{k\in s}|\la_k|- {\vep}\left\|\sum_{k\in s}\la_k y_k\right\|,
\end{align*}
and consequently,
\begin{equation}
\left\|\sum_{k \in s} \la_k y_k\right\|\ge\frac{\vep}{4(1+\vep)}\sum_{k\in s}|\la_k|.
\end{equation}
Now, we have that (c) implies (d) and (d) implies (e) trivially. For the implication (e) implies (f) we use
the following classical result by J. Gillis \cite{Gi}.
\begin{lemma}\label{gillis}
For any $\vep,\de>0$ and $m\in \N$ there is $n:=\mathbf{n}(\vep,\de,m)$ such that whenever
$(\Omega,\Sigma,\mu)$ is a probability space and $(A_i)_{i=1}^n$ is a sequence of $\mu$-measurable sets with
$\mu(A_i)\ge \vep $ for every $1\le i\le n$, there is $s\con \{1,\dots,n\}$ of cardinality $m$ such that
$$\mu(\bigcap_{i\in s}A_i)\ge (1-\de)\vep^m.$$
\end{lemma}
Incidentally, the   counterexample by P. Erd{\H{o}}s and A. Hajnal of the natural generalization of Gillis'
result concerning double-indexed sequences will be crucial for our solution to Question 1 (see Section
\ref{counterexample}).

We pass now to see that (e) implies (f): Fix a sequence $(x_n)_n$ in $A$ converging weakly to $x$ and
$\vep>0$. By (e), we can find a subsequence $(y_n)_n$ of $(x_n)_n$ and  a $\la$-norming set $\mc N$,
$0<\la\le 1$, such that $(y_n)_n$ uniformly-weakly-converges with respect to $\mc N$. Going towards a
contradiction, suppose $(y_n)_n$ does not uniformly weakly-converge to $x$. Fix then $\vep>0$ such that there
are arbitrary large sets in $\mc F_\vep((y_n)_n)$.   In this case we see that then $\mc
F_{\la\vep(1-\de)}((y_n)_n, \mc N)$ has also arbitrary large sets, contradicting our hypothesis.  Set
$z_n:=y_n-x$ for every $n\in \N$. Now given $m\in \N$, let $x^*\in B_{X^*}$ be such that
$$s:=\conj{n\in M}{|x^*(z_n)|\ge \vep} \text{ has cardinality $\ge \mathbf{n}(\frac{\vep\de\la}{2K},\frac12,m)$},$$
where $K:=\sup_n \nrm{z_n}$. By  a standard separation result, there are $f_1,\dots,f_l\in \mc N$ and
$\nu_1,\dots,\nu_l$ such that $\sum_{i=1}^l |\nu_i|\le \la^{-1}$ and
\begin{equation}
\Big|\sum_{i=1}^l \nu_i f_i(z_n)\Big|\ge \vep(1-\frac{\de}2) \text{ for every $n\in s$}.
\end{equation}
Now on $\{1,2,\dots,l\}$ define the probability measure induced by the convex combination
$$\Big(\frac{1}{\sum_{j=1}^l |\nu_j| }|\nu_i|\Big)_{i=1}^l.$$
For each $n\in s$, let
$$A_n:=\conj{j\in \{1,\dots,l\}}{|f_j(z_n)|\ge \vep(1-\de)}.$$
Then, for every $n\in s$ one has that
\begin{align*}
 \vep(1-\frac{\de}{2}) \le \Big|\sum_{j=1}^l \nu_j f_j(z_n)\Big|\le \sum_{j\in A_n}|\nu_j|K +\vep(1-\de).
\end{align*}
Hence,
$$ \mu(A_n)\ge \frac{\de \vep \la}{2K}.$$
By Gillis' Lemma, it follows in particular that there is some $t\con\{1,\dots,l\}$ of cardinality $m$ such
that  $\bigcap_{n\in t}A_n \neq \buit $, so let $j$ be in that intersection. It follows then that
$|f_j(z_n)|\ge \la \vep(1-\de)$ for every $n\in t$, hence $t\in \mc F_{\la\vep(1-\de)}((y_n)_n,\mc N)$.

(f) implies (a) because uniformly weakly-convergent sequences are Ces\`{a}ro-convergent. This finishes the proof.

Hence, Question 1 for weakly-null sequences can be reformulated as follows:
\begin{que-intro}\label{qu2}
Suppose that $(x_n)_n$ is a weakly-null sequence such that some sequence in $\conv(\{x_n\}_n)$ generates an
$\ell_1$-spreading model. Does there exist a subsequence of $(x_n)_n$ generating an $\ell_1$-spreading model?
\end{que-intro}

As a consequence of Theorem \ref{char1} we obtain the following well-known   0-1-law by P. Erd\"{o}s and M.
Magidor \cite{EM}.
\begin{cor}
Every bounded sequence in a Banach space  has a subsequence such that either all its further subsequences are
Ces\`{a}ro-convergent, or none of them.
\end{cor}
To see this, let $(x_n)_n$ be a sequence in a Banach space. If $A:=\{x_n\}_n$ is Banach-Saks, then, by (e)
above, there is a uniformly weakly-convergent subsequence $(y_n)_n$ of $(x_n)_n$, and as we have mentioned
above, every further subsequence of $(y_n)_n$ is Ces\`{a}ro-convergent. Now, if $A$ is not Banach-Saks, then by (b)
there is a weakly-convergent sequence $(y_n)_n$ in $A$ with limit $y$ generating an $\ell_1$-spreading model.
We have already seen that if $(z_n)_n$ is a basic subsequence of $(y_n-y)_n$, then no further subsequence of
it is Ces\`{a}ro-convergent.

We introduce now the Schreier-like spaces, which play an  important role for the Banach-Saks property.
\begin{defn} \label{def-2}
Given a family $\mc F$ on $\N$, we define the Schreier-like norm $\nrm{\cdot}_\mc F$  on $c_{00}(\N)$ as follows.
For each $x\in c_{00}$ let
\begin{equation}\label{def-21}
\|x\|_\mc F=\max\{\|x\|_\infty,\sup_{s\in \mc F}\sum_{n\in s}|(x)_n|\},
\end{equation}
where $(x)_n$ denotes the $n^{\mr{th}}$-coordinate of $x$ in the usual Hamel basis of $c_{00}(\N)$. We define
the Schreier-like space $X_\mc F$ as the completion of $c_{00}$ under the ${\mc F}$-norm.
\end{defn} Note that $X_{\mc F}=X_{\widehat{\mc F}}$ for every family $\mc F$, so the hereditary property of $\mc F$ plays no
role for the corresponding space. It is clear that the unit vector basis $(u_n)_n$ is a 1-unconditional
Schauder basis of $X_\mc F$, and  it is weakly-null if and only if  $\mc F$ is pre-compact. In fact,
otherwise there will be a subsequence of $(u_n)_n$ 1-equivalent to the unit basis of $\ell_1$. So,
Schreier-like spaces will be assumed to be constructed from pre-compact families. It follows then that for
pre-compact families $\mc F$, the space $X_\mc F$ is $c_0$-saturated. This can be seen, for example, by using
Pt\'{a}k's Lemma, or by the fact that $X_\mc F=X_{\widehat{F}} \hookrightarrow C(\widehat{\mc F})$ isometrically,
and the fact that the function spaces $C(K)$ for $K$ countable are $c_0$-saturated, by a classical result of
A. Pelczynski and Z. Semadeni \cite{Pelczynski-Semadeni}.

Observe that the unit basis of the \emph{Schreier space} $X_\mc S$ generates an  $\ell_1$-spreading model, so
no subsequence of it can be Ces\`{a}ro-convergent.   In fact, the same holds for the Schreier-like space $X_\mc F$ of
an arbitrary large family $\mc F$.   However, it was proved by M. Gonz\'alez and J. Guti\'errez in \cite{GG} that
the convex hull of a Banach-Saks subset of the Schreier space $X_\mc S$ is again Banach-Saks. In fact, we
will see in Subsection \ref{generalized-Schreier} that the same holds for the spaces $X_\mc F$ where $\mc F$
is a generalized Schreier family.         Still, a possible counterexample for Question \ref{qu1} has to be a
Schreier like space, as we see from the following characterization.

\begin{thm}\label{erioeiofjioedf}
 The following are equivalent:
\begin{enumerate}
\item[(a)] There is a normalized weakly-null sequence having the Banach-Saks property and  whose convex hull is not  a Banach-Saks set.
\item[(b)] There is a Shreier-like space $X_\mc F$ such that its unit basis $(u_n)_n$ is Banach-Saks and its convex hull is not.
\item[(c)] There is a compact  and hereditary family $\mc F$ on $\N$ such that:
\begin{enumerate}
\item[(c.1)]  $\mc F$ is not large in any $M\con \N$.
\item[(c.2)]  There is a partition $\bigcup_n I_n=\N$ in finite sets $I_n$   a probability measure $\mu_n$
on $I_n$ and $\de>0$ such that the set
\begin{equation}
\label{j4ijirjtf}\mc G_\de^{\bar \mu}(\mc F):=\conj{t\con \N}{\text{there is $s\in \mc F$ such that
 $\min_{n\in t}\mu_n(s\cap I_n)\ge \de$}}
\end{equation}
is large.
\end{enumerate}
\end{enumerate}
\end{thm}
For the proof we need the following useful result.
\begin{lemma}\label{lem-3}
Let $(x_n)_n$ and $(y_n)_n$ be two bounded sequences in a Banach space $X$.
\begin{enumerate}
\item[(a)] If $\sum_n\nrm{x_n-y_n}<\infty$, then $\{x_n\}_n$ is Banach-Saks if and only if $\{y_n\}_n$ is
Banach-Saks.
\item[(b)] $\conv(\{x_n\}_n)$ is a Banach-Saks set if and only if every block sequence in $\conv(\{x_n\}_n)$ has
the Banach-Saks property.
\end{enumerate}
\end{lemma}

\begin{proof} The proof of (a) is straightforward. Let us concentrate in (b):  Suppose that $\conv(\{x_n\}_n)$ is not Banach-Saks, and let $(y_n)_n$ be a sequence in $\conv(\{x_n\}_n)$
without Ces\`{a}ro-convergent subsequences. Write $y_n:=\sum_{k\in F_n}\la_k^{(n)}x_k$, $(\la_k^{(n)})_{k\in
F_n}$ a convex combination, for each $n$. By a Cantor diagonalization process we find $M$ such that
$((\lambda_{k}^{(n))_{k\in \N}})_{n\in M}$ converges pointwise to a (possibly infinite) convex sequence
$(\la_k)_k \in B_{\ell_1}$.   Set $\mu_k^{(n)}:=\la_k^{(n)}-\la_k$ for each $n\in M$.  Then there is an
infinite subset $N\con M$ and a block sequence $((\eta_{k}^{(n)})_{k\in s_n})_{n\in N}$, $\sum_{k\in
s_n}|\eta_k^{(n)}|\le 2$,  such that
\begin{equation}\label{dkfmkdfmskd}
\sum_{n\in N}\sum_{k\in \N}|{\mu_k^{(n)}-\eta_k^{(n)}}|<\infty.
\end{equation}
Setting $z_n:=\sum_{k\in s_n} \eta_k^{(n)}x_k$ for each $n$, it follows from \eqref{dkfmkdfmskd} that
\begin{equation}
\sum_{n\in N}\nrm{y_n-z_n}<\infty.
\end{equation}
By (a), no subsequence of   $(z_n)_{n\in N}$ is Ces\`{a}ro-convergent.  Now set $t_n:=\conj{k\in
s_n}{\eta_k^{(n)}\ge 0}$, $u_n=s_n\setminus t_n$, $z_n^{(0)}:=\sum_{k\in t_n}{\eta_k^{(n)}}$ and
$z_n^{(1)}:=z_n-z_n^{(0)}$. Then, either $\{z_n^{(0)}\}_{n\in N}$ or $\{z_n^{(1)}\}_{n\in N}$ is not
Banach-Saks. So, without loss of generality, let us assume that $\{z_n^{(0)}\}_{n\in N}$ is not Banach-Saks.
Then, using again (a), and by going to a subsequence if needed, we may assume that $\sum_{k\in
t_n}\eta_k^{(n)}=\eta$ for every $n\in N$. It follows that the block sequence $((1/\eta)\sum_{k\in
t_n}\eta_k^{(n)}x_n)_{n\in N}$ in $\conv(\{x_n\}_n)$ does not have that Banach-Saks property.
\end{proof}
\prue[\textsc{Proof of Theorem \ref{erioeiofjioedf}}]
It is clear that (b) implies (a). Let us prove that (c) implies (b). We fix a family $\mc F$ as in (c). We
claim that $X_\mc F$ is the desired Schreier space: Let $(u_n)_n$ be the unit basis of $X_\mc F$, and let
$$\mc N:=\{\pm u_n^*\}_n\cup \conj{\sum_{n\in s}\pm u_n^*}{s\in \mc F},$$
where $(u_n^*)$ is the biorthogonal sequence to $(u_n)_n$. Then
$$\mc F_\vep((u_n),\mc N)=\mc F\cup [\N]^1$$
for every $\vep>0$, so it follows from  our hypothesis (c.1) and Theorem \ref{char1} (d) that $\{u_n\}_n$ is
Banach-Saks. Define now for each $n\in \N$, $x_n:=\sum_{k\in I_n} (\mu_n)_k u_k$. Then
$$\mc F_\de((x_n)_n,\mc N)=\mc G_\de(\mc F)$$
so $\mc F_\de((x_n)_n)$ is large, hence $\{x_n\}_n\con \conv(\{u_n\}_n)$ is not Banach-Saks.

 Finally, suppose that (a) holds and we work to
see that (c) also holds. Let $(x_n)_n$ be a weakly-null sequence in some space $X$ with the Banach-Saks
property but such that $\conv(\{x_n\}_n)$ is not Banach-Saks.  By the previous Lemma \ref{lem-3} (b), we may
assume that there is a block sequence $(y_n)_n$ with respect to $(x_n)_n$  in $\conv(\{x_n\}_n)$ without the
Banach-Saks property. By Theorem \ref{char1} there is some subsequence $(z_n)_n$ of $(y_n)_n$ and $\vep>0$
such that
\begin{equation}
\label{opj4t4rjt44}\mc F_\vep((z_n)_n)\text{ is large}.
\end{equation}
By re-enumeration if needed, we may assume that $\bigcup_n \supp z_n=\N$, where the support is taken with
respect to $(x_n)$. Let
$$\mc F:=\mc F_{\frac{\vep}2}((x_n)_n).$$ On the other hand, since
$(x_n)_n$ is weakly-null, it follows that $\mc F$ is pre-compact, and, since it is hereditary by definition,
it is compact. Again by invoking Theorem \ref{char1} we know  that $\mc F$ is not large in any $M\con \N$.  Now
let $I_n:=\supp z_n$  and let $\mu_n$ be the convex combination with support $I_n$ such that $z_n=\sum_{k\in
I_n}(\mu_n)_k x_k$ for each $n\in \N$. Then $(I_n)_n$ is a partition of $\N$ and  $\mu_n$ is a probability
measure on $I_n$.  We see now that \eqref{j4ijirjtf} holds for $\de:=\vep/2$:  Fix an infinite subset $M\con
\N$, and fix $m\in \N$. By \eqref{opj4t4rjt44}, we can find $x^*\in B_{X^*}$ such that
\begin{equation}
\label{j4o3rp4jr4}
s:=\conj{n\in M}{|x^*(z_n)|\ge \vep }\text{ has cardinality $\ge m$.}
\end{equation}
We claim that $s\in \mc G_{\vep/2}^{\bar \mu}(\mc F)$: Fix $n\in s$, and let $s_n:=\conj{k\in
I_n}{|x^*(x_k)|\ge \vep/2}$ and $t_n:=I_n\setminus s_n$. Then
\begin{align*}
\vep \le x^*(z_n)\le \sum_{k\in s_n}(\mu_n)_k+\sum_{k\in t_n} (\mu_n)_k \frac{\vep}2 \le  \sum_{k\in s_n}(\mu_n)_k+\frac{\vep}2
\end{align*}
hence $\mu_n(s_n)\ge \vep/2$, and so $s\in \mc G_{\vep/2}^{\bar \mu}(\mc F)$.
\fprue

\section{Stability under convex hull: positive results}\label{positive results}

Recall that  a Banach space $X$ is said to have the weak Banach-Saks property if every weakly convergent
sequence in $X$ has a Ces\`aro convergent subsequence. Equivalently, every weakly compact set in $X$ has
\pbs. Examples of Banach spaces with the weak Banach-Saks property but without the Banach-Saks property are
$L^1$ and $ c_0 $ (see \cite{Szlenk}).

The following simple observation provides our first positive result concerning the stability of Banach-Saks sets under convex hulls.

\begin{prop} \label{prop-1}
Let $X$ be Banach space with the weak Banach-Saks property. Then the convex hull of a Banach-Saks subset of
$X$ is also Banach-Saks.
\end{prop}

\begin{proof}
If $A \subseteq X$ has \pbs, then $A$ is relatively weakly compact. Therefore, by Krein-\v{S}mulian's
Theorem, $\conv(A)$ is also relatively weakly compact. Since $X$ has the weak Banach-Saks property, it
follows that $\conv(A)$ has \pbs.
\end{proof}

However, the weak Banach-Saks property is far   from being a necessary condition. For instance, the Schreier
space $X_{{\mc S}}$ does not have the weak Banach-Saks property \cite{Szlenk}, but the convex hull of any
Banach-Saks set is again a Banach-Saks set (see \cite[Corollary 2.1]{GG}). In Section
\ref{generalized-Schreier}, we will see that this result can be extended to generalized Schreier spaces.

%
%

Another partial result is the following.
\begin{prop}\label{Teor-1}
Let $(x_n)_n$ be a sequence in a Banach space $X$ such that every subsequence is Ces{\`a}ro convergent. Then
$\conv(\{x_n\})$ is a Banach-Saks set.
\end{prop}

\begin{proof}
As we mentioned in Section 2, the hypothesis is equivalent to saying that $(x_n)_n$ is uniformly
weakly-convergent to some $x\in X$ \cite[Theorem 1.8]{Mercourakis}. Now, by Lemma \ref{lem-3} (b), it
suffices to prove that every block sequence $(y_n)_n$ with respect to $(x_n)_n$ in $\conv(\{x_n\}_n)$ is
Banach-Saks. Indeed we are going to see that such sequence $(y_n)_n$ is uniformly weakly-convergent to x. Fix
$\vep>0$, and let $m$ be such that
\begin{equation}
\label{i4ijti4jtr}
\mc F_\vep((x_n)_n)\con [\N]^{\le m}.
\end{equation}
We claim that $\mc F((y_n)_n,\vep)\con [\N]^{\le m}$ as well: So, let $x^*\in B_{X^*}$ and define
$s:=\conj{n\in \N}{|x^*(y_n-x)|\ge \vep}$. Using that $\{y_n\}_n\con \conv(\{x_n\}_n)$ we can find for each
$n\in s$,an integer $l(n)\in \N$  such that $|x^*(x_{l(n)}-x)|\ge \vep$.   Since $(y_n)_n$ is a block
sequence with respect to $(x_n)_n$, it follows that $(l(n))_{n\in s}$ is a 1-1 sequence. Finally, since
$\{l(n)\}_{n\in s}\in \mc F_\vep((x_n)_n)$, it follows from \eqref{i4ijti4jtr} that $\#s\le m$.
\end{proof}
It is worth to point out  that the hypothesis and conclusion in the previous proposition are not equivalent:
The unit basis of the space $(\bigoplus_n \ell_1^n)_{c_0}$ is not uniformly weakly-convergent (to 0) but its
convex hull is a Banach-Saks set.

Recall that for a $\sigma$-field $\Sigma$ over a set $\Omega$ and a Banach space $X$, a function $\mu:\Sigma\rightarrow X$ is called a (countably additive) vector measure if it satisfies
\begin{enumerate}
  \item $\mu(E_1\cup E_2)=\mu(E_1)+\mu(E_2)$, whenever $E_1,E_2\in \Sigma$ are disjoint, and
  \item for every pairwise disjoint sequence $(E_n)_n$  in $\Sigma$ we have that $\mu(\cup_{n=1}^\infty E_n)=\sum_{n=1}^\infty \mu(E_n)$ in the norm of
  $X$.
\end{enumerate}

\begin{prop}\label{vectormeasure}
If a Banach-Saks set $A$ is contained in the range of some vector measure, then $\conv(A)$ is also Banach-Saks.
\end{prop}

\begin{proof}
J. Diestel and C. Seifert proved in \cite{Diestel-Seifert76} that every set contained in the range of a vector measure is
Banach-Saks. Although the range of a vector measure $\mu(\Sigma)$ need no be a convex set, by a classical result of I. Kluvanek and G. Knowles
\cite[Theorems IV.3.1 and V.5.1]{KK}, there is always a (possibly different) vector measure $\mu'$ whose range contains the convex hull
of $\mu(\Sigma)$. Thus if a set $A$ is contained in the range of a vector measure, then $\conv(A)$ is also a Banach-Saks set.
\end{proof}

However, there are Banach-Saks sets which are not the range of a vector measure: consider for instance the unit ball of $\ell_p$ for $1<p<2$ \cite{Diestel-Seifert76}.

\subsection{A result for generalized Schreier spaces}\label{generalized-Schreier}
We present here a positive answer to Question 1 for a large class of Schreier-like spaces, the spaces
$X_\al:=X_{\mc S_\al}$ constructed from the generalized Schreier families $\mc S_\al$ for a countable ordinal
number $\al$.

Recall that given two families $\mc F$ and $\mc G$ on $\N$, we define
\begin{align*}
\mc F \oplus \mc G:=&\conj{s\cup t}{s\in \mc G,\, t\in \mc F \text{ and $s<t$}} \\
\mc F\otimes \mc G :=&\conj{s_0\cup \dots \cup s_n}{(s_i) \text{ is a block sequence in $\mc F$ and $\{\min s_i\}_{i\le n}\in \mc G$}},
\end{align*}
where $s<t$ means that $\max s<\min t$.
\begin{defn}
For each countable limit ordinal number $\al$ we fix a strictly increasing sequence $(\be^{(\al)}_n)_n$ such
that $\sup_n \be_n^{(\al)}=\al$.  We define now
\begin{enumerate}
\item[(a)] $\mc S_0:= [\N]^{\le 1}$.
\item[(b)] $\mc S_{\al+1}=\mc S_\al \otimes \mc S$.
\item[(c)] $\mc S_\al:= \bigcup_{n\in \N} \mc S_{\be_n^{(\al)}}\rest [n+1,\infty[$.
\end{enumerate}
\end{defn}
Then each $S_\al$ is a compact, hereditary and spreading family with Cantor-Bendixson rank equal to
$\om^\al$. These families have been widely used in Banach space theory. As an example of their important role
we just mention that given a pre-compact family $\mc F$ there exist an infinite set $M$, a countable ordinal
number $\al$ and $n\in \N$ such that $\mc S_\al \otimes [M]^{\le n}\con \mc F[M]\con \mc S_{\al}\otimes
[M]^{\le n+1} $. It readily follows that every subsequence of the unit basis of $X_\mc F$ has a subsequence
  equivalent to a subsequence of the unit basis of $X_\al$. The main result of this part is the following.

\begin{thm}\label{ioo34iji4jtr}
Let $\alpha$ be a countable ordinal number. $A\con X_\al$   has the Banach-Saks property if and only if $\conv(A)$ has the Banach-Saks property.
\end{thm}
The particular case $\al=0$ is a consequence of the weak-Banach-Saks property of $c_0$ and Proposition
\ref{prop-1}. For $\al\ge 1$ the spaces $X_\al$ are not weak-Banach-Saks. Still, Gonz\'alez and Guti\'errez
proved the case $\al=1$ in \cite{GG}. Implicitly, the case $\al<\om$ was proved by I. Gasparis and D. Leung
\cite{GL} since it follows from their result stating that every seminormalized weakly-null sequence in
$X_{\al}$, $\al<\om$, has a subsequence equivalent to a subsequence of the unit basis of $X_\be$, $\be\le
\al$. We conjecture that the same should be true for an arbitrary countable ordinal number $\al$.

The next can be proved by transfinite induction.
\begin{prop}\label{njkrjggff} Let $\be<\ou$.
\begin{enumerate}
\item[(1)]For every $\al<\be$ there is some $n\in \N$ such that $(\mc S_\al \otimes \mc S)\rest (\N / n)\con \mc S_\be$.
\item[(2)] For every $n\in \N$ there are $\al_0,\dots,\al_n<\be$ such that
$$(\mc S_\al)_{\le n}:=\conj{s\in \mc S_\be}{\min s\le n}\con \mc S_{\al_0}\oplus \cdots \oplus \mc S_{\al_n}.$$
\end{enumerate}\qed
\end{prop}

Fix a countable ordinal number $\al$. We introduce now a property in $X_\al$ that will be used to
characterize the Banach-Saks property for subsets of  $X_\al$.

\begin{defn}
We say that a weakly null sequence $(x_n)_n$ in $X_\al$ is \emph{$<\al$-null} when
$$\text{for every  $\be<\al$   and every $\vep>0$ the set
$\conj{n\in \N}{\nrm{x_n}_{\be}\ge \vep}\text{ is finite}.$}$$
\end{defn}

\begin{prop}\label{iurhtiurt}
Suppose that $(x_n)_n$ is a  bounded sequence in $X_\al$ such that there are $\vep>0$, $\be<\al$ and  a block
sequence $(s_n)_n$ in $\mc S_\be$ such that $\sum_{k\in s_n}|(x_n)_k|\ge \vep$. Then  $\{x_n\}_{n}$ is not
Banach-Saks.
\end{prop}
\prue
 Let $K=\sup_{n}\nrm{x_n}$. Let $\bar{n}\in \N$ be such that $(\mc S_\be \otimes \mc S)\rest [\bar n,\infty[\con \mc S_\al$. Fix a subsequence $(x_n)_{n\in
 M}$.
\begin{claim}\label{ijirjirjgr}
For every $\de>0$ there is a subsequence $(x_n)_{n\in N}$ such that for every $n\in N$ one has that
\begin{equation}\label{iorijir}
 \sum_{m\in N, \, m<n} \max\left\{\sum_{k\in s_n} |(x_m)_k|, \sum_{k\in s_m} |(x_n)_k|\right\}\le \de.
\end{equation}
\end{claim}
The proof of this claim is the following.  Using that $(u_n)_n$ is a Schauder basis of $X_\al$ and that
$(s_n)_n$ is a block, we can find  a
 subsequence $(x_n)_{n\in N}$ such that  for every $n\in N$ one has that
 \begin{equation}
\sum_{m\in N, \, m<n} \sum_{k\in s_n} |(x_m)_k|\le \de.
 \end{equation}
We  color each pair $\{m_0<m_1\}\in [\N]^2$ by
$$c(\{m_0,m_1\})=\left\{\begin{array}{ll}
0 & \text{ if $\sum_{k\in s_{m_0}}|(x_{m_1})_k|\ge \de$}\\
1 & \text{ otherwise}.
\end{array}\right.$$
By the Ramsey Theorem, there is some infinite subset $P\con N$ such that $c$ is constant on $[P]^2$ with
value $i= 0,1$. We claim that $i=1$.  Otherwise, suppose that $i=0$. Let $m_0\in P$, $m_0>\bar n$ be such
that $ m_0 \cdot \de
> K $, and let $m_1\in P$ be such that $t=[m_0,m_1[\cap P$ has cardinality $m_0$. Then $ n_0< m_0\le \min
s_{m_0}$, and hence $s=\bigcup_{m\in t}s_m\in \mc S_\al$. But then,
$$K\ge \nrm{x_{m_1}}\ge \sum_{k\in s}|(x_{m_1})_k|=\sum_{m\in t}\sum_{k\in s_m} |(x_{m_1})_k|\ge \# t \cdot \de >K,$$
a contradiction. Now it is easy to find $P\con N$ such that for every $n\in P$,
 \begin{equation}
\sum_{m\in P, \, m<n} \sum_{k\in s_m} |(x_n)_k|\le \de.
 \end{equation}
Using the Claim \ref{ijirjirjgr}   repeatedly, we can find $N\con M$ such that
$$\sum_{n\in N}\sum_{m\neq n\in N }\sum_{k\in s_m}|(x_n)_k|\le \frac{\vep}2.$$
In other words, $(x_n, \sum_{k\in s_n }\theta_k^{(n)}u_k^*)_{n\in N}$ behaves almost like a biorthogonal
sequence for every sequence of signs $((\theta_k^{(n)})_{k\in s_n})_{n\in N}$. We see now that $(x_n)_{n\in
N}$ generates an $\ell_1$-spreading model with constant $\ge \vep/2$. We assume without loos of generality
that $\bar n<N$. Let $t\in \mc S\rest N$, and let $(a_n)_{n\in t}$ be a sequence of scalars such that
$\sum_{n\in t}|a_n|=1$. Then $s=\bigcup_{n\in t} s_n \in \mc S_\al $, and hence,
\begin{align*}
\nrm{\sum_{n\in t}a_n x_n}\ge &\sum_{k\in s}|(\sum_{n\in t}a_n x_n)_k| =\sum_{n\in t}\sum_{k\in s_n}|(\sum_{m\in t}a_m x_m)_k|\ge\\
\ge &  \sum_{n\in t}|a_n|\sum_{k\in s_n} |(x_n)_k| -
\sum_{n\in t}\sum_{k\in s_n} \sum_{m\in t\setminus \{n\}}|(x_m)_k|\ge \vep  \sum_{n\in t}|a_n|- \frac{\vep}{2}\ge \frac\vep2  \sum_{n\in t}|a_n|.
\end{align*}
\fprue
The following characterizes  the Banach-Saks property of  subsets of $X_\al$.
\begin{prop}\label{odjfiojdijdsfsd}
Let $(x_n)_n$ be a  weakly null sequence in $X_\al$. The following are equivalent:
\begin{enumerate}
\item[(1)] Every subsequence of $(x_n)_n$ has a further subsequence dominated by the unit basis of $c_0$.
\item[(2)] Every subsequence of $(x_n)_n$ has a further norm-null  subsequence  or a subsequence
equivalent to the unit basis of $c_0$.
\item[(3)] $\{x_n\}_n$ is a Banach-Saks set.
\item[(4)] $(x_n)_n$ is $<\al$-null.
\end{enumerate}
\end{prop}
\prue
$(1)\Rightarrow (2) \Rightarrow (3)$ trivially.  (3) implies (4):  Suppose otherwise that $(x_n)_n$ is not
$<\al$-null. Fix $\vep>0$ and $\be<\al$ such that
$$M:=\conj{n\in \N}{\nrm{x_n}_{\be}\ge \vep} \text{ is infinite}.$$
For each $n\in M$, let $s_n\in \mc S_{\be}$ such that $\sum_{k\in s_n} |(x_n)_k|\ge \vep.$ Since $(x_n)_{n\in
\N}$ is weakly-null, we can find $N\con \N$ and $t_n\con s_n$ for each $n\in N$ such that $(t_n)_{n\in N}$ is
a block sequence and $\sum_{k\in t_n} |(x_n)_k|\ge \vep/2.$  Then by Proposition \ref{iurhtiurt},
$\{x_n\}_{n\in N}$ is not Banach-Saks, and we are done.

(4) implies (1). Let $K:=\sup_{n\in \N}\nrm{x_n}$. Let $(x_n)_{n\in M}$ be a subsequence of $(x_n)_{n\in
\N}$. If  $\al=0$, Then $X_\al$ is isometric to $c_0$, and so we are done. Let us suppose that $\al>0$.
 Fix $\vep>0$.
\begin{claim}\label{njnjvnfd}
There is $N=\{n_k\}_k\con M$, $n_k<n_{k+1}$,  such that for every $i<j$ and every $s\in \mc S_{\al}$

$$\text{ if $\sum_{k\in s}|(x_{n_{i}})_k|> \vep /2^{i+1}$, then }\sum_{k\in s}|(x_{n_{j}})_k|\le \frac{\vep}{2^{j}}.$$
\qed
\end{claim}
Its proof is the following: Let $n_0=\min M$. Let $m_0\in \N$ be such that
\begin{equation}
\sum_{k>m_0}|(x_{n_0})_k|\le \frac{\vep}{2}.
\end{equation}
In other words,
\begin{equation}
\conj{s\in \mc S_\al}{\sum_{k\in s}|(x_{n_0})_k|> \frac{\vep}2}\con (\mc S_\al)_{\le m_0}.
\end{equation}
 By Proposition \ref{njkrjggff}   (2) there are $\al_0^{(0)},\dots,\al_{l_0}^{(0)}<\al$ such that
\begin{equation}
(\mc S_\al)_{\le m_0}\con   \mc S_{\al_0^{(0)}}\oplus \cdots \oplus \mc S_{\al_{l_0}^{(0)}}.
\end{equation}
We use that $(x_n)_n$ is $<\al$-null to find $n_1\in M$, $n_1>n_0$, be such that for every $n\ge n_1$ one has
that
\begin{equation}
\nrm{x_n}_{(\mc S_\al)_{\le m_0}}\le \frac{\vep}{2 }.
\end{equation}
Let now $m_1>\max\{n_1,m_0\}$ be such that
\begin{equation}
\sum_{k>m_1}|(x_{n_1})_k|\le \frac{\vep}{4}.
\end{equation}
Then there are $\al_0^{(1)},\dots,\al_{l_1}^{(1)}<\al$ such that
\begin{equation}
\conj{s\in \mc S_\al}{\sum_{k\in s}|(x_{n_1})_k|> \frac{\vep}4}\con  (\mc S_\al)_{\le m_1}\con   \mc S_{\al_0^{(1)}}\oplus \cdots \oplus \mc S_{\al_{l_1}^{(1)}}.
\end{equation}
Let now $n_2\in M$, $n_2>n_1$ be such that for every $n\ge n_2$ one has that
\begin{equation}
\nrm{x_n}_{(\mc S_\al)_{\le m_1}} \le \frac{\vep}{4}.
\end{equation}
In general, suppose defined $n_i$, let   $m_i>\max\{n_i,m_{i-1}\}$ be such that
\begin{equation}
\sum_{k>m_i}|(x_{n_i})_k|\le \frac{\vep}{2^{i+1}}.
\end{equation}
Then,
\begin{equation}
\conj{s\in \mc S_\al}{\sum_{k\in s}|(x_{n_i})_k|> \frac{\vep}{2^{i+1}}}\con (\mc S_\al)_{\le m_i}\con   \mc S_{\al_0^{(i)}}\oplus \cdots \oplus \mc S_{\al_{l_i}^{(i)}},
\end{equation}
for some $\al_0^{(i)},\dots,\al_{l_i}^{(i)}<\al$. Let $n_{i+1}\in M$, $n_{i+1}>n_i$ be such that for all
$n\ge n_{i+1}$ one has that
\begin{equation}\label{jkreirjng}
 \nrm{x_n}_{(\mc S_\al)_{\le m_i}}\le \frac{\vep}{2^{i+1}}.
\end{equation}
We have therefore accomplish the properties we wanted for $N$.

Now fix $N$ as in Claim \ref{njnjvnfd}. Then $(x_n)_{n\in N}$ is dominated by the unit basis of $c_0$. To see
this, fix a finite sequence of scalars $(a_i)_{i\in t}$, and $s\in \mc S_\al$. If $\sum_{k\in
s}|(x_{n_i})_k|\le \vep/2^{i+1} $ for every $i\in t$, then,
\begin{align*}
\sum_{k\in s} |(\sum_{i\in t}a_i x_{n_i})_k|\le \max_{i\in t} |a_i| \cdot \sum_{k\in s}\sum_{i\in t}|(x_{n_i})_k|\le   \max_{i\in t} |a_i| \sum_{i\in t} \frac\vep{2^{i+1}}\le \vep \max_{i\in t}|a_i|.
\end{align*}
Otherwise, let $i_0$ be the first $i\in t$ such that $\sum_{k\in s}|(x_{n_i})_k|> \vep/2^{i+1}$. It follows
from the claim that
\begin{equation}
\sum_{k\in s}|(x_{n_{j}})_k|\le \frac\vep{2^{j}} \text{ for every $i_0<j$}
\end{equation}
Hence,
\begin{align*}
\sum_{k\in s} |(\sum_{i\in t}a_i x_{n_i})_k| \le & \sum_{k\in s}|(\sum_{i\in t, \, i<i_0} a_i x_{n_i})_k|+ |a_{i_0}|\cdot \sum_{k\in s} |(x_{n_{i_0}})|_k + \sum_{k\in s}|(\sum_{i>i_0} a_i x_{n_i})_k|\le \\
\le & \max_{i\in t}|a_i|\sum_{i<i_0}\frac{\vep}{2^{i+1}}+ |a_{i_0}|\nrm{x_{n_{i_0}}} + \max_{i\in t}|a_i| \sum_{i>i_0} \frac{\vep}{2^i} \le
(\vep +K)\max_i |a_i| .
\end{align*}
\fprue

\prue[Proof of Theorem \ref{ioo34iji4jtr}]
Suppose that $A$ is Banach-Saks, and suppose that $(x_n)_n$ is a sequence in $\conv(A)$ without
Ces\`{a}ro-convergent subsequences. Since $\conv(A)$, is relatively weakly-compact,  we may assume that $x_n
\to_n x\in X_\al$ weakly. Let $y_n:=x_n-x$ for each $n\in \N$. Then $(y_n)_n$ is a weakly-null sequence
without
  Ces\`{a}ro-convergent subsequences.   Hence, by Proposition \ref{odjfiojdijdsfsd}, there is some $\vep>0$ and some  $\be>0$ such that
$$M=\conj{n\in \N}{\nrm{y_n}_\be\ge \vep}\text{ is infinite}.$$
For each $n\in M$, let $s_n\in \mc S_\be$ such that
$$\sum_{k\in s_n}|(y_n)_k|\ge \vep.$$
For each $n\in M$, write as convex combination, $x_n=\sum_{a\in F_n}\la_a \cdot a$, where $F_n\con A$ is
finite. Since $(y_n)_n$ is weakly-null, it follows that by going to a subsequence if needed that we may
assume that $(s_n)_n$ is a block sequence.  Let   $n_0$  be such that for all $n\ge n_0$ one has that
$\sum_{k\in s_n}|(x)_k|\le \vep/2$. Hence for every $n\ge n_0$ one has that
\begin{align*}
\vep \le & \sum_{k\in s_n} |(y_n)_k| \le \sum_{k\in s_n}|(x)_k| +\sum_{k\in s_n}\sum_{a\in F_n} \la_a |(a)_k| =  \sum_{k\in s_n}|(x)_k| +\sum_{a\in F_n}\la_a\sum_{k\in s_n}  |(a)_k|\le \\
\le & \sum_{k\in s_n}|(x)_k| + \max_{a\in F_n} \sum_{k\in s_n}  |(a)_k| \le \frac{\vep}2 +  \max_{a\in F_n} \sum_{k\in s_n}|(a)_k|.
\end{align*}
So for each $n\ge n_0$ we can find $a_n\in F_n$ such that $\sum_{k\in s_n}|(a_n)_k|\ge \vep/2$.  Then, by
Proposition \ref{iurhtiurt},  $(a_n)_n$  is not Banach-Saks.
\fprue

\begin{conje} Let $\mc F$ be a compact, hereditary and spreading family on $\N$. Then
the convex hull of any Banach-Saks subset $A\con X_\mc F$   is again Banach-Saks.
\end{conje}


%
%
%
%
%
%
%
\section{A Banach-Saks set whose convex hull is not Banach-Saks}\label{counterexample}
The purpose of this section is to present an example of a Banach-Saks set whose convex hull is not. To do
this, using our characterization in Theorem \ref{erioeiofjioedf}, it suffices to find a special pre-compact
family $\mc F$ as in (c) of that proposition. The requirement of $\mc F$ being hereditary is not essential
here because $X_{\mc F}=X_{\widehat{\mc F}}$.

We introduce now some notions of special interest. In what follows, $I=\bigcup_{n\in \N}I_n$ is a partition
of $I$ into finite pieces $I_n$. A \emph{transversal} (relative to $(I_n)_n$) is an infinite subset $T$ of
$I$ such that $\#(T\cap I_n)\le 1$ for all $n$. By reformulating naturally Theorem \ref{classif1} we obtain
the following.

 \begin{lemma} \label{lem-4} Let $T\con I$ be a transversal and $n\in \N$.
    \begin{enumerate}
    \item[(a)] If ${\mc F}$ is not $n$-large in $T$, then there exist a  transversal $T_0 \subseteq T$ and
    $m\leq n$ such that $\mc F[T_0]=[T_0]^{\leq m}$.
    \item[(b)] If $\mc F$ is not large in $T$ then there is some transversal $T_0\con T$ and $n\in \N$ such
    that $\mc F[T_0]=[T_0]^{\le n}$.
    \item[(c)]If  $\mc F$ is $n$-large in $T$, then there exists a transversal $T_0 \subseteq
    T$ such that $[T_0]^{\le n} \subseteq {\mc F}[T_0]$.
     \end{enumerate}
    \end{lemma}

\begin{defn} \label{def-4}   For every $0<\lambda<1$ and $s \in{\mc F}$ let us define
\begin{enumerate}
    \item[(a)] $s[\lambda]:=\{n\in \mathbb{N}: \#(s\cap I_n) \geq \lambda \# I_n \}$,
    \item[(b)] $s[+]:=\{n\in \mathbb{N}: s\cap I_n \neq\emptyset \}$,
\end{enumerate}
and the families of finite sets of $\mathbb{N}$
\begin{enumerate}
    \item[(c)] $\mc G_\lambda({\mc F}) := \{ s[\lambda] : s\in{\mc F}\}$,
    \item[(d)] $\mc G_+({\mc F}) := \{ s[+] : s\in{\mc F}\}$.
\end{enumerate}
\end{defn}

\begin{prop}\label{prop-5} Suppose that $\mc F$ is a $T$-family on $I$.  For every $0<\lambda<1$ and every sequence of scalars $(a_n)_n$, we have that
\begin{equation}\label{nvvnbkvv}
{\lambda}\Big\|\sum_{n }a_n u_n\Big\|_{\mc G_\la(\mc F)}\le
\max\left\{\Big\|\sum_{n}a_n \Big(\frac{1}{\#I_n}\sum_{j\in I_n}
u_j\Big)\Big\|_{\mc F}, \sup_n |a_n|\right\}\le \Big\|\sum_{n}a_n u_n\Big\|_{\mc G_+(\mc F)}.
\end{equation}
\end{prop}

\begin{proof} For each $n$, set
$$
x_n:=\frac{1}{\# I_n}\sum_{j\in I_n} u_j.
$$
Given $(a_n)_n$, by Definition \ref{def-2}, for every $s\in \mc F$, we have that
\begin{eqnarray*}
 \sum_{k\in s} \Big|\Big(\sum_n a_n x_n\Big)_k\Big|  &=& \sum_{n\in s[+]}\sum_{k\in s\cap I_n} \frac{|a_n|}{\#I_n}=\sum_{n\in s[+]}|a_n|\frac{\#(s\cap I_n)}{\#I_n}\\
   &\leq& \sum_{n\in s[+]}|a_n| \leq\Big\|\sum_{n}a_n u_n\Big\|_{\mc G_+(\mc F)},
\end{eqnarray*}
and
\begin{align*}
\sup_k \left|\left(\sum_n a_n x_n \right)_k \right|\le \sup_n \frac{|a_n|}{\#I_n} \le \sup_n|a_n\le \nrm{\sum_n a_nu_n}_{\mc G_+(\mc F)}.
\end{align*}
This proves the second inequality in \eqref{nvvnbkvv}.
 Now, given $t\in \mc G_\lambda(\mc F)$, let $s\in \mc F$ be such that $s[\lambda]=t \subseteq s[+]$.  Then

$$
\sum_{k\in s}\Big|\Big(\sum_{n}a_n x_n\Big)_k\Big| \ge \sum_{n\in s[\lambda]}\sum_{k\in s\cap I_n}|a_n|\frac{1}{\#I_n}=\sum_{n\in s[\lambda]}|a_n|\frac{\#(s\cap
I_n)}{\#I_n}\ge \lambda \sum_{n\in t}|a_n|.
$$
This proves the first inequality in \eqref{nvvnbkvv}.
\end{proof}
Observe that the use of the $\sup$-norm of $(a_n)_n$ in the middle term of \eqref{nvvnbkvv} can be explained
by the fact that the sequence of averages  $(x_n)_n$ is not always seminormalized, independently of the
family $\mc F$. However, for the families we will consider $(x_n)_n$ will be normalized and 1-dominating the
unit basis of $c_0$, so the term $\sup_n |a_n|$ will disappear in \eqref{nvvnbkvv}.

\begin{defn}\label{ij4tijrigrf}
 A pre-compact family $\mc F$ on $I$ is called a  \emph{$T$-family}   when there is a partition $(I_n)_n$  of $I$ into finite pieces $I_n$ such that
\begin{enumerate}
\item[(a)] $\mc F$ is not large in any $J\con I$.
\item[(b)]  There is $0<\la\le 1$ such that $\mc G_{\la}(\mc F)$ is large in $\N$.
\end{enumerate}
\end{defn}
Observe that the pre-compactness of $\mc F$ follows from (a) above.
\begin{prop} \label{prop-6} Let ${\mc F}$ be a $T$-family on   $I=\bigcup_n I_n$. Then
\begin{enumerate}
    \item[(a)]  the block sequence of averages
$\left({1}/{\#I_n}\sum_{i\in I_n} u_i\right)_n$ is not Banach-Saks in $X_{{\mc F}}$.

    \item[(b)] Every subsequence $(u_i)_{i\in T}$ of $(u_i)_{i\in I}$ has a further subsequence $(u_i)_{i\in
    T_0}$ equivalent to the unit basis of $c_0$. Moreover its equivalence constant is at most the integer
    $n$ such that $\mc F[T_0]=[T_0]^{\le n}$.
\end{enumerate}
\end{prop}
\begin{proof} Set $x_n:= {1}/{\#I_n}\sum_{i\in I_n} u_i$ for each $n\in \N$.
(a): From Theorem \ref{classif1} there is $M\con \N$ such that $[M]^1\con \mc G_\la(\mc F)[M]$. This readily
implies that $\nrm{x_n}_\mc F\ge \la$ for every $m\in M$. Therefore, $(x_n)_{n\in M}$ is a seminormalized
block subsequence of the unit basis $(u_n)_{n}$, and it follows that $(x_n)_{n\in M}$ dominates the unit basis of
$c_0$. From the left inequality in \eqref{nvvnbkvv} in Proposition \ref{prop-5} we have that
$(x_n)_{n\in M}$ also dominates the subsequence $(u_n)_{n\in M}$ of the unit basis of $X_{\mc G_\la(\mc F)}$.
Since $\mc G_\la(\mc F)$ is large, no subsequence of its unit basis is Banach-Sack and therefore $(x_n)_{n}$
is not Banach-Saks.

(b)  Let  $(u_i)_{i\in T}$ be a subsequence of the unit basis of $X_{\mc F}$. Without loss of generality, we
assume that $T$ is a transversal of $I$. Using our hypothesis (a),   the Lemma \ref{lem-4} (b) gives us
another transversal $T_0 \subset T$ and $n\in \N$ such that ${\mc F}[T_0]=[T_0]^{\leq n}$. Then the
subsequence $(u_i)_{i\in T_0}$ is equivalent to the unit basis of $c_0$ and therefore  Ces{\`a}ro convergent
to 0. In fact, for every $s\in {\mc F}$ and for every scalar sequence $(a_j)_{j\in T_0}$
$$
\sum_{i\in s}|a_i| = \sum_{i\in s \cap T_0}|a_i|\leq \max\{ |a_i| : i\in s\cap T_0\} \#(s\cap T_0) \leq n \|(a_i)\|_{\infty}.
$$
On the other hand it is clear that $\|(a_i)\|_{\infty}\leq \|\sum_{i\in T_0}a_ie_i\|_{{\mc F}}$.
\end{proof}

This is the main result.
\begin{thm} \label{ioo4ui4}There is a $T$-family on $\N$. More precisely, for every $0<\vep<1$ there is a partition $\bigcup_n I_n$ of $\N$ in
finite pieces $I_n$ and a pre-compact  family $\mc F$ on $\N$ such that
\begin{enumerate}
\item[(a)] $\mc F$ is not 4-large in any $M\con \N$.
\item[(b)]  $\mc G_{1-\vep}(\mc F)=\mc G_+(\mc F)=\mk S$, the Schreier barrier.
\item[(c)] For every $s\in \mc G_+(\mc F)$ one has that $s\cap I_n=I_n$, where $n$ is the minimal $m$ such that $s\cap I_m\neq
\buit$.
\end{enumerate}
\end{thm}

\begin{cor}
For every $\vep>0$ there is a Schreier-like space $X_{\mc F}$ such that every subsequence of the unit basis
of it has a further subsequence 4-equivalent to the unit basis of $c_0$, yet there is a block sequence of
averages $((1/\#I_n)\sum_{i\in I_n}u_i)_n$  which is  $1+\vep$-equivalent to the unit basis of the Schreier
space $X_\mc S$.
\end{cor}
\begin{proof}
From Proposition \ref{prop-5}, it only rests to see that $\nrm{\sum_n a_n x_n}_\mc F\ge \sup_n |a_n|$, where
$x_n=1/\#I_n \sum_{i\in I_n}u_i$ for every $n\in \N$.  To see this, fix a finite sequence of scalars
$(a_n)_{n\in t}$, and fix $m\in t$. Let  $u\in \mk S$ be such that $\min u=m$ and $u\cap t=\{m\}$, and let
$s\in \mc F$ such that $s[+]=u$. Then, by the properties of $\mc F$, it follows that $s\cap I_m=I_m$, while
$s\cap I_n=\buit$ for $n\in t\setminus\{m\}$.  Consequently,
\begin{align*}
\nrm{\sum_{n\in t}a_n x_n}_\mc F \ge & \sum_{k\in s}\left|\left(\sum_{n\in t}a_n \frac{1}{\#I_n}\sum_{i\in I_n} u_i  \right)_k\right|=|a_m|.
\end{align*}
\end{proof}

The construction of our family as in Theorem \ref{ioo4ui4} is strongly influenced by the following
counterexample of Erd\H{o}s and Hajnal \cite{ErHa} to the natural generalization of Gillis' Lemma
\ref{gillis} to double-indexed sequences of large measurable sets.

\begin{lemma} \label{lem-5}
For every $m\in \N$ and $\vep>0$ there is probability space $(\Om,\Sig,\mu)$ and a sequence $(A_{i,j})_{1\le
i<j\le n}$ with $\mu(A_{i,j})\ge \vep$ for every $1\le i<j\le n$ such that for every $s\con \{1,\dots,n\}$ of
cardinality $m$ one has that
$$\bigcap_{\{i,j\}\in [s]^2}A_{i,j}=\buit.$$
\end{lemma}

\begin{proof}
Given $n,r\in \N$ let $\Om:=\{1,\dots,r\}^n$, and let $\mu$ be the probability counting measure on $r^n$.
Given $1\le i<j\le n$ we define the subset of $n$-tuples
\begin{equation}
\label{oi4oi3u4iu5t4} A_{i,j}^{(n,r)}:=\{(a_l)_{l=1}^{n}\in \{1,\dots,r\}^n \, : \,   a_{i}\neq a_{j}\}.
\end{equation}
 This is the desired
counterexample. In fact,
\begin{enumerate}
\item[(a)] $\# A_{i,j}^{(n,r)}= r^n(1-1/r)$ for every $1\le i<j\le n$, and
\item[(b)] $\bigcap_{\{i,j\}\in [s]^2}A^{(n,r)}_{i,j} =\emptyset $ for every  $s\in [\{1,\dots,n\}]^{r+1}$.
\end{enumerate}
To see (a), given  $1\le i< j \leq n$

$$ \{1,2,\ldots,r\}^n \setminus A_{i,j}^{(n,r)}= \bigcup_{\theta =1}^r \{(a_l)_{l=1}^{n} \in \{1,2,\ldots,r\}^k  \, : \, a_{i}= a_{j}=\theta \}
$$
being the last union disjoint. Since
$$\# \{(a_l)_{l=1}^n\in \{1,2,\ldots,r\}^n : a_{i}=a_{j}=\theta \}=r^{n-2},
$$
it follows that  $\# A_{i,j}^{(n,r)}= r^n(1-1/r)$.   It is easy to see (b) holds since otherwise we would
have found a subset of $\{1,\dots,r\}$ of cardinality $r+1$.
\end{proof}

\begin{proof}[Proof of Theorem \ref{ioo4ui4}]  For practical reasons
we will define such family not in $\N$ but in a more appropriate countable set $I$.  Fix $0< \lambda <1 $.  We define first the disjoint sequence $(I_n)_n$.
 For each $m \in \mathbb{N}, m
\geq 4$, let $r_m$ be  such that
\begin{equation} \label{equac-1}
    \left(1-\frac1{r_m}\right)^{\binom{m-2}{2}} \geq \lambda .
\end{equation}
Let $4\leq m\leq n$ be fixed.  Let
$$
I_{m,n}:= \{1,\dots,r_m\}^{n \times [\{2,\ldots,m-1\}]^2}.
$$
 Let $I_n=\{n\}$ for $n=1,2,3$. For $n\geq 4$ let
$$I_n:=\prod_{4\le m \le n}I_{m,n}=\prod_{4\le m \le n}\{1,\dots,r_m\}^{n \times [\{2,\ldots,m-1\}]^2}.$$
Observe that for $n\neq n'$ one has that $I_n \cap I_{n'}=
    \emptyset$.  Let $I:=\bigcup_n I_n$.     Now, given $4\leq m_0 \leq n$ and  $2\le i_0<j_0\le m_0-1$, let
$$ \pi_{i_0,j_0}^{(n,m_0)} : I_n\rightarrow \{1,2,\ldots,r_{m_0}\}^n
$$
be the natural projection,
$$ \pi_{i_0,j_0}^{(n,m_0)}(\,\big((b_{i,j}^{(l,m)})_{}\big)_{4\le m\le n,\,1\le l\le n ,\, 2\le i<j\le m-1  } \,):=(b^{(l,m_0)}_{i_0,j_0})_{l=1}^n\in  \{1,2,\ldots,r_{m_0}\}^n.$$
We start with the definition of the family ${\mc F}$ on $I$.  Recall that $\mk S:=\conj{s\con \N}{\#s=\min
s}$ is the Schreier barrier. We  define $F:\mk S\to [I]^{<\infty}$ such that $F(u)\con \bigcup_{n\in u}I_n$
and then we will define $\mc F$ as the image of $F$.  Fix $u=\{n_1 <\cdots< n_{n_1}\}\in {\mk S}$:
\begin{enumerate}
\item[(i)] For $u=\{1\}$, let $F(u):= I_1$.
\item[(ii)] For $u:=\{2,n\}$, $2<n$, let $F(u):= I_2\cup I_n$.
\item[(iii)] For $u:=\{3,n_1,n_2\}$, $3<n_1<n_2$, let $F(u):= I_3\cup I_{n_1}\cup I_{n_2}$.
\item[(iv)] For $u=\{n_1,\dots,n_{n_1}\}$ with $3<n_1<n_2<\cdots< n_{n_1}$,  then let
$$F(u)\cap I_{n_{k}}:= I_{n_{k}}  \text{ for } k=1,2,3,$$
\end{enumerate}
and  for $3<k\leq n_1$, let
\begin{equation} \label{equac-2} F(u)\cap I_{n_k}: =
\bigcap_{1<i<j<k}\big(\pi_{i,j}^{(n_k,n_1)}\big)^{-1}\big(A^{(n_k,r_{n_1})}_{n_i ,n_j} \big)
    \end{equation}
Where  the $A$'s are as in \eqref{oi4oi3u4iu5t4}.     Explicitly,
$$ F(u)\cap I_{n_k}= \conj{\big((b_{i,j}^{(l,m)})_{}\big)_{4\le m\le n_k,\,1\le l\le n_k ,\, 2\le i<j\le m-1 } \in I_{n_k} }
{ b_{i,j}^{(n_i,n_1)}\neq b_{i,j}^{(n_j,n_1)} \text{, $1<i<j<k$}
}.
$$
Observe that it follows from \eqref{equac-2} that
\begin{equation}
\pi_{i,j}^{(n_k,n_1)}\big(F(u)\cap I_{n_k} \big) = A_{
n_i ,n_j}^{(n_k,r_{n_1})} \subset \{1,2,\ldots,r_{n_1}\}^{n_k}
\end{equation}
for every $1<i<j<k$.

From the definition of ${\mc F}$ it follows that $u=F(u)[+]$ for every $u \in {\mk S}$. Now, we claim that
given  $u \in {\mk S}$, we have that $u=F(u)[\lambda]$, or, in other words, $\#(F(u)\cap I_n)\ge \la \#I_n$
for every $n\in u$. The only non-trivial case is when $u=\{n_1<\cdots <n_{n_1}\}$ with $n_1>3$, and $n=n_{k}$
is such that $3<k\leq n_1$. It follows from the equality in \eqref{equac-2}, (a) in the proof of Lemma
\ref{lem-5}, and the choice of $r_{n_1}$ in \eqref{equac-1} that

$$
\frac{\#(F(u)\cap I_{n_{k}})}{\#(I_{n_k})} =\prod_{1<i<j<k}
\frac{\#(A^{(n_k, r_{n_1})}_{ n_{i},n_{j}})}{(r_{n_1})^{n_{k}}}=
\prod_{1<i<j<k} \left(1-\frac{1}{r_{n_1}}\right) \ge
\left(1-\frac{1}{r_{n_1}}\right)^{\binom{n_1-2}{2}}\ge \lambda
$$
Summarizing, ${\mc G}({\mc F},\lambda)= {\mc G}({\mc F},+)= {\mk S}$. Thus, ${\mc F}$ satisfies the property
(b)  in Theorem \ref{ioo4ui4}. For the property (a)  we use the following fact.

 \begin{lemma} \label{lem-6}Suppose that $\mc A \subseteq \mk S$ is a subset such that
\begin{enumerate}
\item[(a)] $\min u=\min v= n_1 > 3$ for all $u,v\in \mc A$.
\item[(b)] there are $1<i<j< n_1$  and a set $w\subset \mathbb{N}$ such that
\begin{enumerate}
\item[(b.1)] $\#w \ge  r_{n_1}+2$ and $ n_1<\min w$.
\item[(b.2)] For every $l_1<l_2<\max w$ in $w$ there is $u\in \mc A$
such that $\{n_1,l_1,l_2,\max w\}\subset u$, $\#(u\cap \{1,2,\ldots,l_1\})=i$ and $\#(u\cap
\{1,2,\ldots,l_2\})=j$.
\end{enumerate}
\end{enumerate}
Then
$$
I_{\max w}\cap \bigcap_{u\in  \mc A} F(u)=\emptyset.$$

 \end{lemma}

 \begin{proof}[Proof of Lemma \ref{lem-6}]  Observe that for $l\in u$,
 $\#(u\cap\{1,2,\ldots,l\})=i$ just means that $l$ is the
 $i^{\text{th}}-$element of $u$. For every couple $\{l_1 <l_2\}\in [w\setminus \{\max w\}]^2$, take $u_{l_1 ,l_2}\in {\mc A}$ satisfying the
condition of (b.2). Since, $u_{l_1,l_2}=\{n_1<\cdots<n_i=l_1<\cdots<n_j=l_2<\cdots<\max w<\cdots\leq
n_{n_1}\}$, it follows from the equality in \eqref{equac-2} that
$$
\pi_{i,j}^{(\max w,n_1)}\big(F(u_{l_1 ,l_2})\cap I_{\max w} \big) =  A_{l_1,l_2}^{(\max w,r_{n_1})}.
$$
Hence
\begin{align*}
 \pi_{i,j}^{(\max w,n_1)}(I_{\max w} \cap \bigcap_{u\in {\mc A}}F(u)) \subseteq &
\bigcap_{\{l_1,l_2\}
\in [w\setminus \{\max w\}]^2}  \pi_{i,j}^{(\max w,n_1)}( I_{\max w} \cap F(u_{l_1,l_2}))=  \\
=& \bigcap_{\{l_1,l_2\}
\in [w\setminus \{\max w\}]^2} A_{l_1,l_2}^{(\max w,r_{n_1})}
= \emptyset
\end{align*}
where the last equality follows from (b) in the proof of  Lemma \ref{lem-5}, since $\#w\geq r_{n_1} +2$.

\end{proof}

We continue with the proof property (a) of $\mc F$ in Theorem \ref{ioo4ui4}.  Suppose otherwise that there
exists a transversal $T$ of $I$ such that ${\mc F}$ is $4$-large in $T$. By Lemma
  \ref{lem-4} (c), there exists $T_0 \subseteq T$ such that $[T_0]^4 \subseteq {\mc
  F}[T_0]$. For every $k \in T_0$, $n(k)$ denotes the unique integer $m$ for which $k\in I_{m}$. It is easy to see that if
$k_1,k_2 \in T_0$ with $k_1 <k_2$, then $n(k_1)<n(k_2)$. Now, for each $t=\{k_0<k_1<k_2<k_3\}$ in $[T_0]^4$,
let us choose $U(t)\in \mk S$ such that
$${t}\subset F(U({t})).$$
Observe that $\{n(k_0),n(k_1),n(k_2), n(k_3)\}\subset U({t})$, and hence $\#U({t})\le n(k_0)$. Now, let
$$\text{$\bar
k :=\min  T_0$ and $\bar n:=n(\bar k)$.}$$
Define the coloring $\Theta:[T_0\setminus \{\bar k\} ]^3\to [\{1,2,\ldots,n(\bar k)\}]^3$ for each
$t=\{k_1<k_2<k_3\}$ in $T_0\setminus \{\bar k\}$ as
$$
\Theta(t)=\big(\#(U(\{\bar k\}\cup t)\cap \{1,\ldots, n(k_1)\}),\#(U(\{\bar k\}\cup t)\cap \{1,\ldots,n(k_2)\}), \min U(\{\bar k\}\cup t)\big).
$$
By the Ramsey theorem, there exist $1<i<j<n_1\leq n(\theta)$ and $T_1 \subseteq T_0 \setminus \{\theta\}$
such that $\Theta$ is constant on $T_1$ with value $\{i,j,n_1\}$. Choose $k_1< \dots <k_{r_{\bar n}+2}$ in
$T_1$, and  set
$$\mc A:=\{  U(\{\bar k,k_{l_1},k_{l_2},k_{r_{\bar n}+2}\}) : 1\le l_1 <l_2< r_{\bar n}+2 \}.
$$
Notice that $\mc A$ fulfills the hypothesis of Lemma \ref{lem-6}
with respect to the set $w= \{n(t_{l_1}): 1\le l_1 \le
r_{n(\theta)}+2\} $, and therefore
\begin{equation}
I_{n(t_{r_{n(\theta)}+2})}\cap \bigcap_{u\in \mc A}
F(u)=\emptyset,
\end{equation}
which contradicts the fact that
$$k_{r_{\bar n}+2} \in I_{n(k_{r_{\bar n}+2})}\cap \bigcap_{u\in \mc A} F(u).$$
The family $\mc F$ clearly has property (c) from the statement of Theorem \ref{ioo4ui4} by construction. This
finishes the proof of the desired properties of $\mc F$.
\end{proof}

A similar analysis will be used now to prove that the closed linear span of the sequence  $$x_n=\frac{1}{\#
I_n}\sum_{j\in I_n}u_j$$ is not a complemented subspace of $X_\mc F$.  Let $(x_n^*)_n$ denote the sequence of
biorthogonal functionals to $(x_n)_n$ on $[x_n]^*$.

\begin{prop}
If $T:[u_k]_k\rightarrow [x_n]_n$ is a linear mapping such that
$$
\lim_{k\rightarrow\infty}\langle x_{n(k)}^*,Tu_k\rangle\neq0,
$$
then $T$ cannot be bounded. In particular, there does not exist a projection $P:X_\mc F\rightarrow [x_n]_n$.
\end{prop}

\prue
Let us suppose that $T$ is bounded. Since  $\lim_{k\rightarrow\infty}\langle x_{n(k)}^*,Tu_k\rangle\neq0$, let $\alpha>0$ be such that $|\langle x_{n(k_j)}^*,Tu_{k_j}\rangle|\geq\alpha$ for every $j\in \mathbb{N}$. Moreover, since $(u_k)_k$ is weakly null, up to equivalence we can assume that $(Tu_{k_j})_j$ are disjoint blocks with respect to $(x_n)_n$.

By Proposition \ref{prop-6}(b), passing to a further subsequence it holds that $(u_{k_j})_j$ is 3-equivalent to the unit basis of $c_0$. Now, let $0<\lambda\leq1$ such that $\mc G_\lambda(\mc F)=\mk S$, and take $n_0>\frac{3\|T\|}{\alpha\lambda}$. Let $u\in \mk S$ with $\min u=n_0$. We have
$$
3\geq\Big\|\sum_j u_{k_j}\Big\|\geq\frac{1}{\|T\|}\Big\|\sum_j Tu_{k_j}\Big\|_{X_\mc F}\geq\sum_{i\in F(u)}|\langle u^*_i,\sum_jTu_{k_j}\rangle|\geq\frac{n_0\alpha\lambda}{\|T\|}.
$$
This is a contradiction with the choice of $n_0$.
\fprue

\begin{rem} The Cantor-Bendixson rank of a $T$-family must be infinite. To see this, observe that if $f:I\to
J$ is finite-to-one \footnote{ $f:I\to J$ is finite-to-one when $f^{-1}\{j\}$ is finite for every $j\in J$.}
then $f$ preserves the rank $\ro(\mc F)$ of pre-compact families $\mc F$ in $I$. Since $n(\cdot):I\to \N$,
$n(i)=n$ if and only if $i\in I_n$ is finite-to-one and since $n(\mc F)=\conj{\{n(i)\}_{i\in s}}{s\in \mc
F}=\mc G_+(\mc F)\supseteq \mc G_\la(\mc F)$ is large, it follows that $\ro(n(\mc F))=\ro(\mc F)$ is
infinite. In this way our $T$-family $\mc F$ in Theorem \ref{ioo4ui4} is minimal because $\ro(\mc
F)=\ro(n(\mc F))=\ro(\mk S)=\om$.
\end{rem}

\subsection{A reflexive counterexample}
There is a reflexive counterpart of our example $X_\mc F$.  Indeed we are going to see that the Baernstein
space $X_{\mc F,2}$ for our family $\mc F$ is such space. It is interesting to note that the corresponding
construction $X_{\mc S,2}$ for the Schreier family $\mc S$ was used by A. Baernstein II in \cite{Baernstein}
to provide the first example of a reflexive space without the Banach-Saks property. This construction was
later generalized by C. J. Seifert in \cite{Sei} to obtain $X_{\mc S,p}$.
\begin{defn}
Given a pre-compact family $\mc F$, and given $1\le p\le \infty$, one defines on $c_{00}(\N)$ the norm
$\nrm{x}_{\mc F,p}$ for a vector $x\in c_{00}(\N)$ as follows:
\begin{equation}
\nrm{x}_{\mc F,p}:= \sup\conj{\nrm{(\nrm{E_i x}_{\mc F})_{i=1}^n}_p}{E_1<\dots<E_n,  n\in \N}
\end{equation}
where $E_1<\dots<E_n$ are finite sets and $E x$ is the natural projection on $E$ defined by $Ex:= \mathbbm
1_E \cdot x$. Let $X_{\mc F,p}$ be the corresponding completion of $(c_{00},\nrm{\cdot}_{\mc F,p})$.
\end{defn}
Again, the unit Hamel basis of $c_{00}$ is a 1-unconditional Schauder basis of $X_{\mc F,p}$. Notice also
that this construction generalizes the Schreier-like spaces, since $X_{\mc F,\infty}=X_\mc F$.
\begin{prop}\label{ij4i5otjirjtr}
The space $X_{\mc F,p}$ is $\ell_p$-saturated. Consequently, if $1<p<\infty$, the space $X_{\mc F,p}$ is
reflexive.
\end{prop}
\begin{proof}
The case $p=\infty$ was already treated when we introduced the Schreier-like spaces after Definition
\ref{def-2}. So, suppose that $1\le p<\infty$.
\begin{claim}\label{ioo4rji4j}
Suppose that $(x_n)_n$ is a normalized block sequence of $(u_n)_n$. Then
\begin{equation}
\nrm{\sum_n a_n x_n}_{\mc F,p}\ge \nrm{(a_n)_n}_p.
\end{equation}
\end{claim}
To see this, for each $n$, let $(E_i^{(n)})_{i=1}^{k_n}$ be a block sequence of finite sets such that
\begin{equation}
1=\sum_{i=1}^{k_n}  \nrm{E_i^{(n)}x_n}_{\mc F}^p.
\end{equation}
Without loss of generality we may assume that $\bigcup_{i=1}^{k_n}E_i^{(n)}\con \supp x_n$, hence
$E_{k_n}^{(n)}<E_{1}^{(n+1)}$ for every $n$. Set $x=\sum_n a_n x_n$. It follows that
\begin{align*}
(\nrm{\sum_n a_n x_n}_{\mc F,p})^p \ge \sum_n \sum_{i=1}^{k_n} \nrm{E_i^{(n)} x}_\mc F^p=\sum_n |a_n|^p.
\end{align*}
This finishes the proof of Claim \ref{ioo4rji4j}. It follows from this claim that $c_0\not\hookrightarrow
X_{\mc F,p}$. Fix now a normalized block sequence $(x_n)_n$ of $(u_n)_n$ and  $\vep>0$. Let $(\vep_n)_n$ be
such that $ \sum_n \vep_n^p\le \vep/2$, $\vep_n>0$ for each $n$. Since  $c_0\not\hookrightarrow X_{\mc F,p}$
and since $X_\mc F$ is $c_0$-saturated, we can find a $\nrm{\cdot}_{\mc F,p}$-normalized block sequence $(y_n)_n$
of $(x_n)_n$ such that
\begin{equation}
\nrm{y_n}_\mc F\le \vep_n.
\end{equation}
\begin{claim}
For every sequence of scalars $(a_n)_n$ we have that
\begin{equation} \label{kjhjohoiuhiu}
\nrm{(a_n)_n}_p\le \nrm{\sum_n a_n y_n}_{\mc F,p}\le (1+\vep)\nrm{(a_n)_n}_p.
\end{equation}
\end{claim}
Once this is established, we have finished the proof of this proposition. The first inequality in
\eqref{kjhjohoiuhiu} is consequence of Claim \ref{ioo4rji4j}. To see the second one, fix a block sequence
$(E_i)_{i=1}^l$ of finite subsets of $\N$. For each $n$, let $B_n:=\conj{j\in \{1,\dots,l\}}{E_j x_n\neq
\buit}$, and for $n$ such that $B_n\neq \buit$, let $i_n:=\min B_n$, $j_n:=\max B_n$.  Observe that
$i_n,j_n\in B_m$ for at most one $m\neq n$.  Then, setting  $y=\sum_n a_n y_n$,
\begin{align*}
\sum_{i=1}^l \nrm{E_i y}_\mc F^p = & \sum_{i\in \bigcup_n B_n}\nrm{E_i y}_\mc F^p \le \sum_n \sum_{i\in B_n}
 \nrm{E_i y}_\mc F^p\le |a_1|^p\sum_{i\in B_1}\nrm{E_i y_1}_\mc F^p +\nrm{E_{j_1}y}_\mc F^p + \\
+& \sum_{n\ge 2} \left(|a_n|^p\sum_{i\in B_n}\nrm{E_i y_n}_\mc F^p+\nrm{E_{i_n} y}_\mc F^p+ \nrm{E_{j_n} y}_\mc F^p     \right)\le \\
\le &  \sum_n |a_n|^p\nrm{y_n}_{\mc F,p}^p + 2\max_n|a_n|^p \sum_n \vep_n^p \le (1+\vep)\sum_n |a_n|^p.
\end{align*}

\end{proof}

\begin{prop}
 Given $0<\la< 1$, let $\mc F$ be a $T$-family for $\la$ as in Theorem \ref{ioo4ui4} with respect to some $\bigcup_n I_n$.
Then
\begin{enumerate}
\item[(a)] Every subsequence of the unit basis of $X_{\mc F,p}$ has a further subsequence 6-equivalent to the
unit basis of $\ell_p$.
\item[(b)] The sequence of averages
$$\left( \frac{1}{\#I_n}\sum_{i\in I_n}u_i   \right)_n$$
is $\la$-equivalent to the unit basis of the Seifert space $X_{\mc S,p}$.
\end{enumerate}
\end{prop}

\begin{proof}
(a): Fix a subsequence $(u_n)_{n\in M}$ of $(u_n)_n$ and let $(u_n)_{n\in N}$ be a further sequence of it
such that $\mc F[N]\con [N]^{\le 3}$. Fix also a sequence of scalars $(a_n)_{n\in N}$ such that $x=\sum_{n\in
N} a_n u_n\in X_{\mc F,p}$.   Given a finite subset $E\con \N$ we obtain that
\begin{equation}
\nrm{E x}_\mc F\le 3\max_{n \text{ is such that $E x_n\neq 0$}}|a_n|.
\end{equation}
Now given a block sequence $(E_i)_{i=1}^l$ of finite subsets of $\N$, and given $i=1,\dots,l$, let
$A_i=\conj{n\in N}{E_i x_n\neq 0}$ and let $B:=\conj{i\in \{1,\dots l\}}{A_i\neq \buit}$. Then we obtain that
\begin{align*}
\sum_{i=1}^l \nrm{E_i x}_\mc F^p \le 3\sum_{i\in B} (\max_{n\in A_i} |a_n|)^p \le 6 \sum_n |a_n|^p
\end{align*}
the last inequality because $A_i\cap A_j=\buit$ if $i<j$ are not consecutive in $B$, and if $i<j$ are
consecutive, then $\#(A_i\cap A_j)\le 1$.    The other inequality is proved in the Claim \ref{ioo4rji4j} of
Proposition \ref{ij4i5otjirjtr}.

Let us prove (b):  First of all, observe that by definition we have that $X_{\mc S,p}=X_{\mk S,p}$.  Set
$x_n:=(1/\#I_n)\sum_{i\in I_n}u_i$ for each $n\in \N$, and fix a sequence of scalars $(a_n)_n$.  Set also
$$x=\sum_n a_n x_n \text{ and } u=\sum_{n}a_n u_n.$$
Let $(E_i)_{i=1}^l$ be a block sequence of finite subsets of $\N$ such that
\begin{equation}
\nrm{\sum_{n}a_n u_n}_{\mk S,p}^p=\sum_{i=1}^l \nrm{E_i u}_{\mk S}^p.
\end{equation}
For each $i=1,\dots,l$, let $t_i\in \mk S$ be such that $\nrm{E_i u}_{\mk G}=\sum_{n\in t_i\cap E_i}|a_n|$.
For each $i=1,\dots,l$ let $s_i\in \mc F$ be such that $s_i[\la]=t_i$, and set $F_i:=\bigcup_{n\in E_i} I_n$.
Notice that $(F_i)_{i=1}^l$ is a block sequence of finite subsets of $\bigcup_n I_n=\N$. Then
\begin{align*}
\nrm{\sum_n a_n x_n}_{\mc F,p}^p \ge & \sum_{i=1}^l \nrm{F_i(\sum_n a_n x_n)}_\mc F^p= \sum_{i=1}^l \nrm{\sum_{n\in E_i} a_n x_n}_\mc F^p \ge \\
\ge &  \sum_{i=1}^l \left(\sum_{k\in s_i}|(\sum_{n\in E_i} a_n x_n)_k|\right)^p \ge \sum_{i=1}^l (\la\sum_{n\in E_i\cap t_i}|a_n|)^p=\la^p\nrm{\sum_{n}a_n u_n}_{\mk S,p}^p.
\end{align*}
For the other inequality, let $(F_i)_{i=1}^l$ be a block sequence such that
\begin{equation}
\nrm{\sum_{n}a_n x_n}_{\mc F,p}^p=\sum_{i=1}^l \nrm{F_i x}_{\mc F}^p.
\end{equation}
For each $i=1,\dots,l$, let $s_i\in \mc F$ be such that $\nrm{F_i x}_{\mc F}=\sum_{k\in s_i} |(F_i x)_k|$,
and $E_i:=\conj{n\in \N}{F_i\cap I_n\neq \buit}$. Then, setting $t_i:=s_i[+]\in \mk S$, we have that
\begin{equation}
\nrm{F_i x}_{\mc
F,p} = \sum_{n\in s_i[+]\cap E_i}|a_n|\frac{\#((s_i \cap F_i)\cap I_n)}{\#I_n}\le \sum_{n\in s_i[+]}|(E_i u)_n|\le \nrm{
E_i u}_\mk S.
\end{equation}
Since $(E_i)_{i=1}^l$ is a block sequence it follows that
\begin{align*}
\nrm{\sum_n a_n u_n}_{\mk S,p}^p\ge &\sum_{i=1}^l \nrm{E_i u}_{\mk S}^p \ge \sum_{i=1}^l \nrm{F_i x}_{\mc F}^p=\nrm{\sum_{n}a_n x_n}_{\mc F,p}^p.
\end{align*}
\end{proof}

There is another, more general, approach to find a reflexive counterexample to Question 1. This can be done
by considering  the interpolation space  $\De_p(W,X)$, $1<p<\infty$, where $W$ is the closed absolute convex
hull of a Banach-Saks subset of $X$ which  it is not Banach-Saks itself.

 Recall that  given a  convex, symmetric and bounded subset $W$   of a Banach space $X$, and $1<p<\infty$,   one defines the
Davis-Figiel-Johnson-Pelczynski \cite{Da-Fi-Jo-Pel} interpolation space $Y:=\De_{p}(W,X)$ as the space
$$\conj{x\in X}{\nrm{x}_Y<\infty},$$
where
$$\nrm{x}_Y:=\nrm{(|x|_n)_n}_p$$ and where for each $n$,
$$|x|_n:=\inf\conj{\la>0}{\frac{x}{\la}\in 2^{n} W +\frac1{2^n}B_X}.$$
The key is the following.
\begin{lemma}
 A subset $A$ of $W$ is a Banach-Saks subset of $X$ if
and only if $A$ is a Banach-Saks subset of $Y:=\De_{p}(W,X)$.
\end{lemma}
\prue
Fix $A\con W$, and set $Y:=\De_{p}(W,X)$ Since the identity $j: Y\to X$ is a bounded operator, it follows
that if $A$ is a Banach-Saks subset of $Y$ then $A=j(A)$ is also a Banach-Saks subset of $X$.

Now suppose that $A$ is a Banach-Saks subset of $X$.  Going towards a contradiction, we fix a  weakly
convergent sequence $(x_n)_n$ in $A$  with limit $x$ generating an $\ell_1$-spreading model. Let $\de$
witnessing that, and set $y_n:=x_n-x\in 2 W$ for each $n$. Observe that it follows from the definition that
\begin{enumerate}
\item[(a)] For every $\la>0$ and every $\vep>0$ there  is $n_0$ such that for every $x\in \la W$ we have that
 $\sum_{n>n_0}|x|_n^p\le \vep$.
\end{enumerate}
Since $A$ is Banach-Saks in $X$, we assume without loss of generality that the sequence $(y_n)_n$ is
uniformly weakly-convergent (to 0).  Observe that then
\begin{enumerate}
\item[(b)] For every $\vep>0$ there is $n$ such that if $\#s=n$, then $\nrm{\sum_{n\in s}y_n}_X\le \vep \#s$.
\end{enumerate}
Consequently,
\begin{enumerate}
\item[(c)] For every $\vep>0$ and $r$ there is $m$ such that if $\#s=m$, then $\sum_{n\le r}|\sum_{k\in s}y_k|^p \le \vep$.
\end{enumerate}
Now let $k\in \N$ be such that $k^{1/p}<\de k$, and $\vep>0$ such that $k^{1/p}+\vep<\de k$. Using (a) and
(c) above we can find finite sets $s_1<\dots s_n$ such that
\begin{enumerate}
\item[(d)] $s=\bigcup_{i=1}^n s_i\in \mc S$.
\item[(e)]  Setting $z_i:=(1/\#s_i)\sum_{k\in s_i}y_k$ for each $i=1,\dots,k$, then there is a block sequence $(v_i)_{i=1}^k$ in $\ell_p$ such that
 $\nrm{v_i}_p\le 1$, $i=1,\dots,k$, and such that
\begin{equation}
\nrm{(|z_1+\dots+z_k|_n)_n-(v_1+\dots+v_k)}_p\le \vep.
\end{equation}
\end{enumerate}
It follows then from (d),  (e) and the fact that $(y_n)_n$ generates an $\ell_1$-spreading model with
constant $\de$  that
\begin{align*}
\de k \le \nrm{\frac{1}{k}\sum_{i=1}^k z_i}_Y \le \nrm{v_1+\dots+v_k}_p+\vep \le k^{\frac1p}+\vep <\de k,
\end{align*}
a contradiction.
\fprue
Let now $X:=X_\mc F$ where $\mc F$ is a $T$-family, let $W$ be the closed absolute convex hull of the unit basis $\{u_n\}_n$ of $X_\mc F$
\begin{prop}
The interpolation space $Y:=\De_{p}(W,X_\mc F)$, $1<p<\infty$, is a reflexive space with a weakly-null
sequence  which is a Banach-Saks subset of $Y$, but its convex hull is not.  \qed
\end{prop}



\bibliographystyle{amsplain}

\providecommand{\bysame}{\leavevmode\hbox
to3em{\hrulefill}\thinspace}
\providecommand{\MR}{\relax\ifhmode\unskip\space\fi MR }
\providecommand{\MRhref}[2]{%
  \href{http://www.ams.org/mathscinet-getitem?mr=#1}{#2}
} \providecommand{\href}[2]{#2}


\end{document}